\documentclass[12pt]{amsart}
\usepackage{amsthm}
\usepackage{amssymb}
\usepackage{amsmath}
\usepackage{amsfonts}
\usepackage{graphicx} 
\usepackage{mathtools}
\usepackage{faktor} 
\usepackage{tikz}
\usepackage{ytableau,tikz,varwidth}
\usetikzlibrary{calc}
\usepackage{tikz-cd}
\usepackage{enumitem}
\usepackage{amsrefs}
\usepackage{latexsym,hyperref,caption,subcaption,comment}

\numberwithin{equation}{section}
\newtheorem{theorem}{Theorem}[section]

\newtheorem{lemma}[theorem]{Lemma}
\newtheorem{prop}[theorem]{Proposition}
\theoremstyle{definition}
\newtheorem{definition}[theorem]{Definition}
\newtheorem{conjecture}[theorem]{Conjecture}
\newtheorem{example}[theorem]{Example}
\newtheorem{remark}[theorem]{Remark}

\DeclarePairedDelimiter\floor{\lfloor}{\rfloor}

\def\opn#1#2{\def#1{\operatorname{#2}}} 
\opn\Cl{Cl} \opn\pdim{pdim} \opn\Im{Im} \opn\Ker{Ker} \opn\ini{in} \opn\ann{ann} \opn\depth{depth}  \opn\Spec{Spec}  \opn\supp{supp} \opn\lcm{lcm} \opn\lex{lex}  \opn\diam{diam}  \opn\height{height}

\title{On the $(S_2)$-condition of edge rings for cactus graphs}

 \author{Rodica Dinu}
\address{%
	University of Konstanz, Fachbereich Mathematik und Statistik, Fach D 197 D-78457 Konstanz, Germany, and Simion Stoilow Institute of Mathematics of the Romanian Academy, Calea Grivitei 21, 010702, Bucharest, Romania}
	\email{rodica.dinu@uni-konstanz.de}

\author{Nayana Shibu Deepthi}
\address{Department of Pure and Applied Mathematics, Graduate School of Information Science and Technology, Osaka University, Suita, Osaka 565-0871, Japan}
\email{nayanasd@ist.osaka-u.ac.jp}


\begin{document}

\begin{abstract}

A cactus graph is a connected graph in which every block is either an edge or a cycle. In this paper, we will examine cactus graphs where all the blocks are $3$-cycles, i.e., triangular cactus graphs, of diameter $4$.
Our main focus is to prove that the corresponding edge ring of this family of graphs is not normal and satisfies Serre's condition $(S_2)$.  We use a criterion due to Katth\"an for non-normal affine semigroup rings.
\end{abstract}

\maketitle

\section{Introduction}\label{sec:introduction}

The study of the edge rings and edge polytopes arising from finite connected graphs is of interest from both combinatorics and commutative algebra perspectives. Being inspired by the article of Ohsugi and Hibi \cite{OH}, we are interested in studying the properties of the edge rings associated with finite graphs. Some interesting results on the edge rings can be found in \cites{HN, svv, HHKO, HMT, Nayana}.

Let $G$ be a finite connected graph on the vertex set $[d]=\{1, 2, \dots, d\}$ with the edge set $E(G)=\{e_1,e_2, \dots, e_n\}$. We will always assume that $G$ is simple, i.e., $G$ is without loops and multiple edges.
We consider the polynomial ring $\mathbb{K}[\mathbf{t}]:=\mathbb{K}[t_1, \dots, t_d]$ in $d$ variables over a field $\mathbb{K}$. For an edge $e=\{i, j\}\in E(G)$, we define the quadratic monomial $\mathbf{t}^e:=t_it_j$. The subring of $\mathbb{K}[\mathbf{t}]$ generated by $\mathbf{t}^{e_1}, \dots, \mathbf{t}^{e_n}$ is called the \textit{edge ring} of $G$, denoted by $\mathbb{K}[G]$. 
By \cite{OH}*{Proposition 1.3}, the edge ring $\mathbb{K}[G]$ has dimension $d$ if $G$ is not a bipartite graph. For an edge $e=\{i, j\}\in E(G)$, we define $\rho(e):=\mathbf{e}_i+ \mathbf{e}_j$, where $\mathbf{e}_1, \dots, \mathbf{e}_d$ are the canonical unit coordinate vectors of $\mathbb{R}^d$. We consider the affine semigroup generated by $\rho(e_1), \dots, \rho(e_n)$, which we denote by $S_G$. Let $\mathbb{A}_G:=\{\rho(e) : e\in E(G)\}$. We may regard the edge ring $\mathbb{K}[G]$ as the affine semigroup ring of $S_G$, for which the cone $\mathcal{C}_G:=\mathbb{Q}_{\geq 0}\mathbb{A}_G$ plays an essential role in its study.

The edge ring $\mathbb{K}[G]$ is normal if and only if $\mathbb{Z}_{\geq 0}\mathbb{A}_G=\mathbb{Q}_{\geq 0}\mathbb{A}_G \cap \mathbb{Z}\mathbb{A}_G$, see \cite{BH}*{Section 6.1}. The problem of the normality of $\mathbb{K}[G]$ was studied by Ohsugi and Hibi in \cite{OH}, and by Simis, Vasconcelos, and Villareal in \cite{svv}. 
They gave a combinatorial criterion for normality in terms of the graph $G$. It turns out that the odd cycle condition in classical graph theory gives a characterization of normal edge rings, see \cite{OH}*{Theorem 2.2} and \cite{svv}*{Theorem 1.1}.
By \cite{Hochster}*{Theorem 1}, it is known that if $\mathbb{K}[G]$ is normal, then $\mathbb{K}[G]$ is Cohen--Macaulay. Using a general characterization for normality \cite{BH}*{Theorem 2.2.22}, it follows that $\mathbb{K}[G]$ is normal if and only if $\mathbb{K}[G]$ satisfies Serre's conditions $(R_1)$ and $(S_2)$.  A criterion in terms of the combinatorics of the graph $G$ for edge rings to satisfy Serre's condition $(R_1)$ was given in \cite{HK}*{Theorem 2.1}. 
In addition, according to a general result due to Trung and Hoa \cite{TrungHoa}*{Theorem 4.1}, and Sch\"afer and Schenzel \cite{SS}*{Theorem 6.3}, Serre's condition $(S_2)$ is a necessary condition for $\mathbb{K}[G]$ to be Cohen--Macaulay.

We are interested in understanding the $(S_2)$-condition for $\mathbb{K}[G]$. We recall here that, for a Noetherian ring $A$, the $(S_2)$-condition requires that $\depth(A_{\mathfrak{p}})\geq \inf\{\height(\mathfrak{p}), 2\}$ for any prime ideal $\mathfrak{p}$. 
Not much is known in general about the $(S_2)$-condition for non-normal edge rings. In the paper \cite{Nayana}, the second author provided a class of graphs for which their associated edge rings are non-normal but satisfy the $(S_2)$-condition. It was also shown that this class of graphs is the smallest with respect to those properties: adding new edges breaks the non-normality of the edge ring or violates the $(S_2)$-condition. Higashitani and Kimura~\cite{higashitani2018necessary} provided a combinatorial criterion to check if an edge ring satisfies the $(S_2)$-condition, which was actually used in the paper \cite{Nayana}. However, no other results are known in the literature about the $(S_2)$-condition (and Cohen--Macaulay property) for non-normal edge rings. In this article, we show that for the class of triangular cactus graphs of diameter 4, the associated edge ring is non-normal and it satisfies the $(S_2)$-condition. Recall that a triangular cactus graph is a connected graph in which every block is a cycle of length 3. We will use a result of Katth\"an \cite{Katt}, which shows that Serre's conditions are hidden in the geometry of the set of holes of the affine semigroup $S_G$.

We present the structure of the article. Section~\ref{sec:prelim} is devoted to recalling fundamental definitions, notations, and results concerning classical graph theory, edge rings, and tools from Katth\"an's paper that will be needed. The key ingredient will be  Theorem~\ref{thm:S2holes}, which says that, for a $d$-dimensional affine semigroup $S$, its affine semigroup ring $\mathbb{K}[S]$ satisfies Serre's condition $(S_2)$ if and only if every family of holes of $S$ is of dimension $d - 1$. In Section~\ref{sec:Cgraphs}, we focus on the class of triangular cactus graphs of diameter 4 and state in Theorem~\ref{thm:cactusS2} that, in this case, the edge ring $\mathbb{K}[G]$ is not normal and satisfies the $(S_2)$-condition. Section~\ref{sec:proof} is entirely devoted to the proof of Theorem~\ref{thm:cactusS2}. The description of the supporting hyperplanes of $\mathcal{C}_G$ depends on the regular vertices and the fundamental sets of $G$, see Theorem~\ref{supphyperplanes}. We shall consider certain vertices and shall analyze, based on their configuration, whether they are regular or non-regular cutpoints in $G$. Hence we distinguish the triangular cactus graphs of diameter 4 into two types, and in both cases, we look for the fundamental sets. In addition, the normalization of $S_G$, denoted $\overline{S_G}$, can be expressed in terms of exceptional pairs in $G$, due to Theorem~\ref{normalization}. Hence, another goal will be to understand the exceptional pairs in both types. In Proposition~\ref{prop:hole1} and Proposition~\ref{prop:hole2}, we show that, in both types, $\overline{S_G}\setminus S_G$ decomposes as a union of $(d-1)$-dimensional families of holes of $S$. Using Theorem~\ref{thm:S2holes}, we conclude our main result. In Section~\ref{sec:conclude}, we present some conclusions. We conjecture that the edge ring of a triangular cactus graph satisfies the $(S_2)$-condition, and we believe that it might even be Cohen--Macaulay.

\section{Preliminaries}\label{sec:prelim}
We recall in this section some fundamental definitions and results from finite graphs theory, edge rings, and an important criterion regarding Serre's condition $(S_2)$ for the affine semigroup ring, which will be of use for our paper.
\subsection{Graph theory} Let $G$ be a finite simple graph on the vertex set $V(G)=[d]$, where $d\geq 2$, and $E(G)=\{e_1, e_2, \dots, e_n\}$, the set of edges of $G$. A {\em path} between any two vertices is a sequence of distinct edges that joins the two vertices. 
A path from vertex $u$ to $v$ is the {\em shortest path} if there is no other path from $u$ to $v$ with a lower length.
The {\em distance} between two vertices is the length of the shortest path between those two vertices.
The shortest path length between a graph's most distant vertices is known as its {\em diameter} and for the graph $G$, let us denote it as $\diam{G}$. A cycle of length $n$ will be called an {\em $n$-cycle}. A cycle has a \textit{chord} if there is a pair of vertices that are adjacent, but not along the cycle.     
A \textit{minimal} cycle in $G$ is a cycle with no chord. An \textit{odd cycle} in a graph is a cycle whose length is odd.

Let $C$ and $C'$ be two minimal cycles of $G$ with $V(C)\cap V(C')=\emptyset$. Then a \textit{bridge} between $C$ and $C'$ is an edge $e=\{i, j\}$ of $G$ with $i\in V(C)$ and $j\in V(C')$.
A pair of odd cycles $(C, C')$ is called an \textit{exceptional} pair if $C$ and $C'$ are minimal odd cycles in $G$ such that $V(C)\cap V(C')=\emptyset$ and there is no bridge connecting them.

We say that a graph $G$ satisfies the \textit{odd cycle condition} if, for any two odd cycles $C$ and $C'$ of $G$, either $V(C)\cap V(C')\neq \emptyset$ or there exists a bridge between $C$ and $C'$. Equivalently, the graph $G$ has no exceptional pairs.

Let $U \subset V(G)$. The \textit{induced subgraph} of $G$ on $U \subset [d]$ is the subgraph $G_U$ of $G$ with $V(G_U)=U$ and $E(G_U)=\{e\in E(G) : e\subset U\}$. For $v\in V(G)$, we denote $G\setminus v$ as the induced subgraph of $G$ on the vertex set $V(G\setminus v)=V(G)\setminus \{v\}$. For $T\subset V(G)$, we denote the set of its neighbor vertices by  
\[
N_{G}(T):= \{v\in V(G) : \{v, w\}\in E(G) \text{ for some } w\in T\}.
\]
A non-empty subset of vertices  $T \subset V(G)$ is called \textit{independent} if no edge of $G$ is of the form $e=\{i, j\}$ with $i\in T$ and $j\in T$. For an independent set $T\subset V(G)$, we define a \textit{bipartite graph induced by $T$} in the following way: its set of vertices is $T\cup N_{G}(T)$ and its set of edges is given by $\{\{v,w\}\in E(G): v\in T, w\in N_{G}(T)\}$.

\begin{definition}
Let $G$ be a finite connected simple graph and $V(G)$ its set of vertices. A vertex $v\in V(G)$ is called \textit{regular} in $G$ if every connected component of $G\setminus v$ contains at least one odd cycle.
A non-empty set $T \subset V(G)$ is called \textit{fundamental} in $G$ if the following conditions are satisfied:
\begin{enumerate}
\item $T$ is an independent set;
\item the bipartite graph induced by $T$ is connected;
\item either $T\cup N_{G}(T)=V(G)$ or every connected component of the graph $G_{V(G)\setminus (T\cup N_{G}(T))}$ contains at least one odd cycle.
\end{enumerate}
\end{definition}
Note that a regular vertex is not the same as a fundamental set with one element.

\begin{example} 
To illustrate that a regular vertex differs from a single element fundamental set, consider the simple graph shown in Figure~\ref{fig:eg}. 
The regular vertices of the graph are $v_{2},v_{3},v_{4}$ and $v_{5}$.
In contrast, there is only one single element fundamental set for the graph, which is  $\{v_{1}\}$, consisting of a vertex that is not regular.
\end{example}

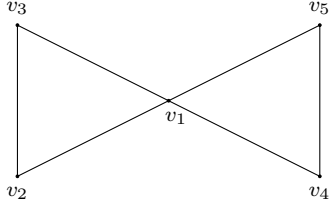
\begin{figure}[ht]
\centering
\begin{tikzpicture}
\draw[black, thin] (0,1) -- (2,2) -- (0,3)-- cycle;
\draw[black, thin] (2,2) -- (4,1) -- (4,3)-- cycle;
\filldraw [black] (2,2) circle (0.5pt);
\filldraw [black] (2.1,2) node[anchor=north] {{\tiny $v_{1}$}};
\filldraw [black] (0,1) circle (0.5pt);
\filldraw [black] (0,1) node[anchor=north] {{\tiny $v_{2}$}};
\filldraw [black] (0,3) circle (0.5pt);
\filldraw [black] (0,3) node[anchor=south] {{\tiny $v_{3}$}};
\filldraw [black] (4,3) circle (0.5pt);
\filldraw [black] (4,3) node[anchor=south] {{\tiny $v_{5}$}};
\filldraw [black] (4,1) circle (0.5pt);
\filldraw [black] (4,1) node[anchor=north] {{\tiny $v_{4}$}};
\end{tikzpicture}
\caption{Differentiating regular vertex and single element fundamental set}
\label{fig:eg}
\end{figure}

\subsection{\texorpdfstring{A characterization of Serre's condition $(S_2)$}{Serre}}

For an arbitrary affine semigroup $S\subset\mathbb{Z}^{d}_{\geq 0}$, let $\{s_1,\dots,s_t\}$ be the minimal
finite subset of $S$ such that $ S =\{\sum_{i=1}^{t}z_{i} s_{i} \colon z_i \in \mathbb{Z}_{\geq 0}\}$.
Then $\{s_1,\dots,s_t\}\subset S$ is called the {\em minimal generating system} of $S$. 
Note that, for $\mathbb{B} \subset \mathbb{R}^d$, we denote $\mathbb{B} S=\{\sum_{i=1}^{t} b_{i}s_{i} : b_i \in \mathbb{B}\}$. 
Let $\mathbb{Z} S$ be the free abelian group generated by $S$ and $\mathbb{Q}_{\geq 0}S$ be the rational polyhedral cone generated by $S$. 
Let $\overline{S} = \mathbb{Q}_{\geq 0}S \cap \mathbb{Z} S$ be the {\em normalization} of $S$.
Analogous to the definition of normality of the affine semigroup $S$, we say that $S$ is normal if $S= \overline{S}$ holds. (See, e.g., \cite{BH}*{Section 6.1}.) 
 The set $F\subset S$ is said to be a {\em face} of $S$, if the following holds: $ s, s^{\prime} \in S,\ s + s^{\prime} \in F$ if and only if $s\in F$ and $s^{\prime}\in F$.
The dimension of a face $F$ is defined to be the rank of the free abelian group $\mathbb{Z} F$.
Throughout our study, we consider only {\em positive} affine semigroups, i.e., the minimal face of $S$ is $\{0\}$.

In the article \cite{Katt}, Katth\"{a}n gave a geometric description for the set of holes $\overline{S}\backslash S $, and this is connected to the ring-theoretical features of $S$, where the affine semigroups considered were not necessarily positive. We recall it here as it will be of great importance for our results.

\begin{theorem}[{\cite{Katt}*{Theorem 3.1}}]\label{thm:holes}
Let $S$ be an affine semigroup. 
Then there exists a (not-necessarily disjoint) decomposition 
\begin{equation}\label{eq:holes}
  \overline{S} \backslash   S = \bigcup\limits_{i=1}^{l}(s_i + \mathbb{Z} F_i) \cap \mathbb{Q}_{\geq 0} S
\end{equation}
with $s_i \in \overline{S}$ and faces $F_i$ of $S$. If no $s_i + \mathbb{Z} F_i$ can be omitted from the union, then the decomposition is unique.
\end{theorem}

The set $s_i + \mathbb{Z} F_i$ in \eqref{eq:holes} is called a $j$-{\em dimensional family of holes} of $S$, where $j=\dim F_{i}$. 

\begin{theorem}[{\cite{Katt}*{Theorem 5.2}}]\label{thm:S2holes}
Let $S$ be an affine semigroup of dimension $d$. 
Then the affine semigroup ring $\mathbb{K}[S]$ satisfies Serre's condition $(S_2)$ if and only if every family of holes of $S$ is of dimension $d - 1$.
\end{theorem}

\subsection{Edge rings}
Let $G$ be a finite connected simple graph on the vertex set $V(G)=[d]$ with the edge set $E(G)=\{e_1, e_2, \dots, e_n\}$. We consider $\mathbb{K}[\mathbf{t}]:=\mathbb{K}[t_1, \dots, t_d]$, the polynomial ring in $d$ variables over the field $\mathbb{K}$. If $e=\{i, j\}\in E(G)$, then we define $\mathbf{t}^e \in \mathbb{K}[\mathbf{t}]$ for the quadratic monomial $t_{i}t_{j}$. We write $\mathbb{\mathbb{K}}[G]$ for the toric ring $\mathbb{K}[\{\mathbf{t}^e : e\in E(G)\}]$ and we call it the \textit{edge ring} of $G$. For more details, we refer to \cites{OH, HHO}.

We consider the canonical unit coordinate vectors of $\mathbb{R}^d$, denoted by $\mathbf{e}_1,  \dots, \mathbf{e}_d$. For an edge $e=\{i, j\}$, we define $\rho(e):= \mathbf{e}_i+\mathbf{e}_j$. Now we consider $\mathbb{A}_G:=\{\rho(e) : e\in E(G)\}$ and let $S_G$ be the affine semigroup generated by $\rho(e_1), \dots, \rho(e_n)$, so $S_G=\mathbb{Z}_{\geq 0}\mathbb{A}_G$. One can regard the edge ring $\mathbb{K}[G]$ as the affine semigroup ring of $S_G$. Consider the convex rational polyhedral cone spanned by $S_G$ in $\mathbb{Q}^d$, denoted by $\mathcal{C}_G$. We may assume that $\mathcal{C}_G$ is $d$-dimensional. The facets of $\mathcal{C}_G$ are given by the intersection of the half-spaces defined by the supported hyperplanes of $\mathcal{C}_G$ which are presented in the following result.

\begin{theorem}[\cite{OH}*{Theorem 1.7}]\label{supphyperplanes}
Let $G$ be a finite connected simple graph with $V(G)=[d]$ containing at least one odd cycle. Then all the supporting hyperplanes of $\mathcal{C}_G$ are given by: 
\begin{enumerate}
\item $\mathcal{H}_v= \{(x_1, \dots, x_d)\in \mathbb{R}^d : x_v=0\}$, where $v$ is a regular vertex in $G$;
\item $\mathcal{H}_T=\{(x_1, \dots, x_d)\in \mathbb{R}^d : \sum_{i\in T}x_i=\sum_{j\in N_{G}(T)}x_j\}$, where $T$ is a fundamental set in $G$.
\end{enumerate}
\end{theorem}
 We denote by $F_v$ and $F_T$ respectively the facets of $\mathcal{C}_G$ corresponding to the supporting hyperplanes $\mathcal{H}_v$ and $\mathcal{H}_T$.

\begin{theorem}[\cite{OH}*{Theorem 2.2}]\label{normalization}
The normalization of the edge ring $\mathbb{K}[G]$ can be expressed as
\[
\overline{S_G}=S_G + \mathbb{Z}_{\geq 0}\{\mathbb{E}_C + \mathbb{E}_{C'} : (C, C') \text{ is an exceptional pair in } G \},
\]
where for any odd cycle $C$, we define $\mathbb{E}_C:=\sum_{i\in V(C)}\mathbf{e}_i$.
\end{theorem}

We observe that, $2\big(\mathbb{E}_{C}+\mathbb{E}_{C^{\prime}}\big)=\big(\sum_{e\in E(C)}\rho(e)+\sum_{e^{\prime}\in E(C^{\prime})}\rho(e^{\prime})\big)\in S_{G}.$
The following result characterizes the normality of edge rings.
\begin{theorem}[\cite{OH}*{Theorem 2.2},\cite{svv}*{Theorem 1.1}]\label{thm:normal}   
Let G be a finite simple graph. Then the edge ring $\mathbb{K}[G]$ is normal if and only if $G$ satisfies the odd cycle condition.
\end{theorem}

Observe that if $G$ satisfies the odd cycle condition, then $\mathbb{K}[G]$ is a Cohen--Macaulay ring, due to a result of Hochster~\cite{Hochster}*{Theorem 1}.

\section{On triangular cactus graphs}\label{sec:Cgraphs}

Let $G$ be a finite connected simple graph. 
A vertex $v\in V(G)$ is called a {\em cutpoint} if the subgraph $G\backslash v$ of $G$ has more connected components than that of $G$.
Recall that a vertex $v$ is {\em regular} in $G$ if every connected component of $G\backslash v$ contains at least one odd cycle.
If a cutpoint $v$ is regular in $G$, then $v$ is called a {\em regular cutpoint}.
We say that a connected graph without a cutpoint is {\em non-separable}. 
A maximal non-separable subgraph of $G$ is called a {\em block} of the graph $G$. 

A {\em cactus graph} is a connected graph in which every block is either an edge or a cycle. 
The cactus graph whose blocks are all $n$-cycles is defined as $n$-{\em cactus graph}. For more details on cactus graphs, we refer to \cite{RinaldoCactus}.
In this study, we will focus only on $3$-cactus graphs, which will be called {\em triangular cactus} graphs.

The $3$-cycles are triangular cactus graphs of diameter $1$.
The friendship graphs, formed from a collection of $3$-cycles joined together at a single common vertex, are triangular cactus graphs of diameter $2$. For edge ring related studies for this family of graphs, see \cite{HN}.
A generic illustration of a triangular cactus graph of diameter $3$ is shown in Figure~\ref{fig:cac3}.
Therefore, we observe that any triangular cactus graph of diameter less than or equal to 3 satisfies the odd cycle condition, and by Theorem~\ref{thm:normal}, it has a normal Cohen--Macaulay edge ring.

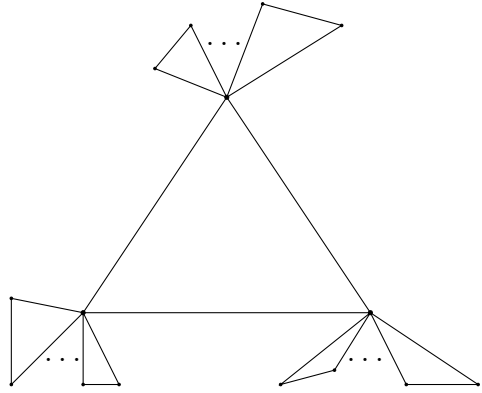
\begin{figure}[ht]
\centering
\begin{tikzpicture}[scale=0.95]
\draw[black, thin] (2,1) -- (6,1) -- (4,4)-- cycle;
\filldraw [black] (2,1) circle (0.8pt);
\draw[black, thin] (2,1) -- (1,0) -- (1,1.2)-- cycle;
\filldraw [black] (1,0) circle (0.5pt);
\filldraw [black] (1,1.2) circle (0.5pt);
\draw[black, thin] (2,1) -- (2,0) -- (2.5,0)-- cycle;
\filldraw [black] (2,0) circle (0.5pt);
\filldraw [black] (2.5,0) circle (0.5pt);
\path (1.5,0.2) -- node[auto=false]{$\dots$} (2,0.5);
\filldraw [black] (6,1) circle (0.8pt);
\draw[black, thin] (6,1) -- (4.75,0) -- (5.5,0.2)-- cycle;
\filldraw [black] (4.75,0) circle (0.5pt);
\filldraw [black] (5.5,0.2) circle (0.5pt);
\draw[black, thin] (6,1) -- (6.5,0) -- (7.5,0)-- cycle;
\filldraw [black] (6.5,0) circle (0.5pt);
\filldraw [black] (7.5,0) circle (0.5pt);
\path (5.93,0.2) -- node[auto=false]{$\dots$} (6,0.5);
\filldraw [black] (4,4) circle (0.8pt);
\draw[black, thin] (4,4) -- (3,4.4) -- (3.5,5)-- cycle;
\filldraw [black] (3,4.4) circle (0.5pt);
\filldraw [black] (3.5,5) circle (0.5pt);
\draw[black, thin] (4,4) -- (4.5,5.3) -- (5.6,5)-- cycle;
\filldraw [black] (4.5,5.3) circle (0.5pt);
\filldraw [black] (5.6,5) circle (0.5pt);
\path (4,5) -- node[auto=false]{$\dots$} (4,4.5);
\end{tikzpicture}
\caption{A general triangular cactus graph of diameter $3$}
\label{fig:cac3}
\end{figure}

We focus on the non-normal edge rings associated to triangular cactus graphs that satisfy $(S_2)$-condition.
A broader objective of this study is to prove the following conjecture:

\begin{conjecture}\label{cnj:cactus}
The edge ring associated to a triangular cactus graph of diameter $\geq 4$ satisfies $(S_{2})$-condition. 
\end{conjecture}

We will exclusively study the family of triangular cactus graphs whose diameter is $4$ and prove Conjecture~\ref{cnj:cactus} for this specific family.
Let $\mathbf{G}$ be the triangular cactus graph of diameter $4$. For an illustration of $\mathbf{G}$ in general, see Figure~\ref{fig:cactgen}.

\begin{figure}[ht]
\centering
\begin{tikzpicture}
\path (3,0.2) -- node[auto=false]{$\dots$} (7,2.5);
\draw[black, thin] (3,0) -- (5,2) -- (2,3)-- cycle;
\filldraw [black] (3,0) circle (0.8pt);
\filldraw [black] (2,3) circle (0.8pt);
\filldraw [black] (3.5,4.2) circle (0.8pt);
\filldraw [black] (6.5,4.2) circle (0.8pt);
\filldraw [black] (8,3) circle (0.8pt);
\filldraw [black] (7,0) circle (0.8pt);
\filldraw [black] (3,0) node[anchor=west] {$x_i$};
\filldraw [black] (2,3) circle (0.8pt);
\path (3.4,3) -- node[auto=false]{$\vdots$} (4.6,2.8);
\draw[black, thin] (5,2) -- (3.5,4.2) -- (6.5,4.2)-- cycle;
\filldraw [black] (5,2) circle (1pt);
\filldraw [black] (5,2) node[anchor=north] {$w$};
\path (7,2.5) -- node[auto=false]{$\vdots$} (5,3.2);
\draw[black, thin] (5,2) -- (8,3) -- (7,0)-- cycle;
\draw[black, thin] (3,0) -- (3.2,-0.7) -- (3.5,-0.6)-- cycle;
\filldraw [black] (3.2,-0.7) circle (0.5pt);
\filldraw [black] (3.2,-0.6) node[anchor=north] {{\tiny $y_{i,2}$}};
\filldraw [black] (3.5,-0.6) circle (0.5pt);
\filldraw [black] (3.4,-0.6) node[anchor=west] {{\tiny $y_{i,1}$}};
\path (2.5,-0.5) -- node[auto=false]{$\dots$} (3.3,-0.5);
\draw[black, thin] (3,0) -- (2.3,-0.9) -- (1.7,-0.5)-- cycle;
\filldraw [black] (2.3,-0.9) circle (0.5pt);
\filldraw [black] (2.3,-0.9) node[anchor=north] {{\tiny $y_{i,2t-1}$}};
\filldraw [black] (1.7,-0.5) circle (0.5pt);
\filldraw [black] (1.7,-0.5) node[anchor=south] {{\tiny $y_{i,2t}$}};
\draw[black, thin] (3,0) -- (1.9,0.4) -- (2.4,0.6)-- cycle;
\path (2.5,-0.2) -- node[auto=false]{$\vdots$} (2,0.4);
\filldraw [black] (1.9,0.4) circle (0.5pt);
\filldraw [black] (1.97,0.4) node[anchor=east] {{\tiny $y_{i,2s_{i}-1}$}};
\filldraw [black] (2.4,0.6) circle (0.5pt);
\filldraw [black] (2.4,0.5)  node[anchor=south] {{\tiny $y_{i,2s_{i}}$}};
\draw[black, thin] (2,3) -- (1.2,2.3) -- (1.4,2)-- cycle;
\filldraw [black] (1.2,2.3) circle (0.5pt);
\filldraw [black] (1.4,2) circle (0.5pt);
\draw[black, thin] (2,3) -- (1.2,3.3) -- (1.4,3.7)-- cycle;
\filldraw [black] (1.2,3.3) circle (0.5pt);
\filldraw [black] (1.4,3.7) circle (0.5pt);
\path (1,3) -- node[auto=false]{$\vdots$} (2,3);
\draw[black, thin] (3.5,4.2) -- (2.6,4) -- (3,3.52)-- cycle;
\filldraw [black] (2.6,4) circle (0.5pt);
\filldraw [black] (3,3.52) circle (0.5pt);
\draw[black, thin] (3.5,4.2) -- (3.2,4.8) -- (3.8,5)-- cycle;
\filldraw [black] (3.2,4.8) circle (0.5pt);
\filldraw [black] (3.8,5) circle (0.5pt);
\path (3,4.5) -- node[auto=false]{$\vdots$} (3.5,4.5);
\draw[black, thin] (6.5,4.2) -- (6.6,3.4) -- (7,3.52)-- cycle;
\filldraw [black] (6.6,3.4) circle (0.5pt);
\filldraw [black] (7,3.52) circle (0.5pt);
\draw[black, thin] (6.5,4.2) -- (6.8,4.8) -- (7.5,4.8)-- cycle;
\filldraw [black] (6.8,4.8) circle (0.5pt);
\filldraw [black] (7.5,4.8) circle (0.5pt);
\path (6.7,4) -- node[auto=false]{$\vdots$} (7,4.5);
\draw[black, thin] (8,3) -- (7.8,3.4) -- (8.3,3.5)-- cycle;
\filldraw [black] (7.8,3.4) circle (0.5pt);
\filldraw [black] (8.3,3.5) circle (0.5pt);
\draw[black, thin] (8,3) -- (8.1,2.6) -- (8.5,2.4)-- cycle;
\filldraw [black] (8.1,2.6) circle (0.5pt);
\filldraw [black] (8.5,2.4) circle (0.5pt);
\path (8,2.7) -- node[auto=false]{$\vdots$} (8.5,3.7);
\draw[black, thin] (7,0) -- (6,-0.5) -- (6.5,-0.5)-- cycle;
\filldraw [black] (6,-0.5) circle (0.5pt);
\filldraw [black] (6.5,-0.5) circle (0.5pt);
\draw[black, thin] (7,0) -- (7.34,-0.5) -- (7.6,-0.3)-- cycle;
\filldraw [black] (7.34,-0.5) circle (0.5pt);
\filldraw [black] (7.6,-0.3) circle (0.5pt);
\path (6.36,-0.6) -- node[auto=false]{$\dots$} (7.6,0);
\end{tikzpicture}
\caption{A general form of a triangular cactus graph of diameter $4$}
\label{fig:cactgen}
\end{figure}
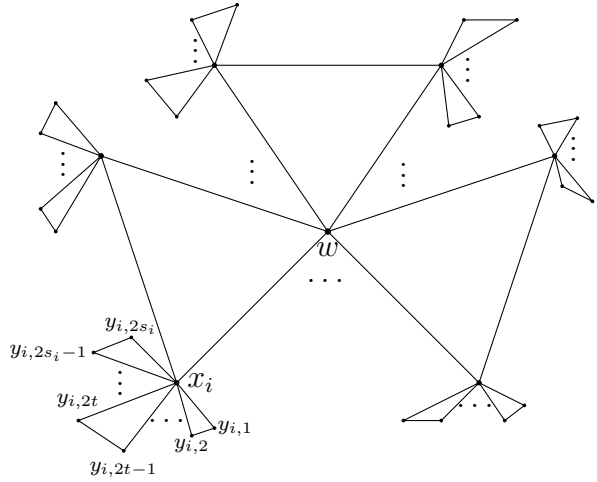

The main theorem of our study is as follows.

\begin{theorem}\label{thm:cactusS2}
Let $\mathbf{G}$ be a triangular cactus graph with $\diam{ \mathbf{G}} = 4$. 
Then the edge ring $\mathbb{K}[\mathbf{G}]$ is not normal, and satisfies the $(S_{2})$-condition.  
\end{theorem}


\section{Towards the proof of Theorem~\ref{thm:cactusS2}}\label{sec:proof}

\hspace{0.5cm}This section centers on the proof of Theorem~\ref{thm:cactusS2}. Let $\mathbf{G}$ be the triangular cactus graph of diameter $4$ and let $|V(\mathbf{G})|=d$. Here, we will study the two types of graph $\mathbf{G}$ and collect the necessary results about them to finally prove our main theorem.

Throughout this paper, we label the vertices of a triangular cactus graph $\mathbf{G}$ as follows.
Let $n$ be a positive integer. The graph is constructed with a central vertex $w$ that is adjacent to $2n$ vertices, denoted as $x_{1}, x_{2}, \dots , x_{2n}$. Each of these vertices is part of a $3$-cycle containing $w$.
For each vertex $x_{i}$ (where $1\leq i\leq 2n$), we define $s_{i}\geq 0$ as the number of additional $3$-cycles that share $x_{i}$ but do not contain $w$.
The vertices of these cycles are labeled as $y_{i,1}, y_{i,2}, \dots , y_{i,2s_{i}}$.
A crucial property of our graph $\mathbf{G}$ is that, given the specified diameter of $4$, there exist at least two non-adjacent vertices $x_{i}$ and $x_{j}$ such that both $s_{i}$ and $s_{j}$ are at least $1$.
Consequently, an arbitrary $3$-cycle consisting of $x_{i}$ that excludes $w$ can be represented as $\{x_{i},y_{1,2t-1},y_{i,2t}\}$ for some $t\in [s_{i}]$ (see Figure~\ref{fig:cactgen}).


Recall that, by Theorem~\ref{supphyperplanes}, for any regular vertex $v$ and fundamental set $T$ of $\mathbf{G}$, we denote by $F_{v}$ and $F_{T}$ the facets of $\mathcal{C}_{\mathbf{G}}=\mathbb{Q}_{\geq 0}\mathbb{A}_{\mathbf{G}}$ corresponding to the hyperplanes $\mathcal{H}_{v}$ and $\mathcal{H}_{T}$, respectively.

The exceptional pairs in $\mathbf{G}$ are of the form $(C,C^{\prime})$ where $C=\{x_{p},y_{p,2t-1},y_{p,2t}\}$ and $C^{\prime}=\{x_{q},y_{q,2t^{\prime}-1},y_{q,2t^{\prime}}\}$ for some $p,q\in [2n]$, $t\in[s_{p}]$, and $t^{\prime}\in[s_{q}]$ such that $x_{p}\neq x_{q}$ and $\{x_{p}, x_{q}\}\notin E(\mathbf{G})$.
An illustration of an exceptional pair in $\mathbf{G}$ is shown in Figure~\ref{fig:cactexpl}. 

\begin{figure}[ht]
\centering
\begin{tikzpicture}
\path (3,0.2) -- node[auto=false]{$\dots$} (7,2.5);
\draw[black, thin] (3,0) -- (5,2) -- (2,3)-- cycle;
\filldraw [black, thick] (3,0) circle (0.8pt);
\filldraw [black] (2,3) circle (0.8pt);
\filldraw [black] (3.5,4.2) circle (0.8pt);
\filldraw [black, thick] (6.5,4.2) circle (0.8pt);
\filldraw [black] (7,0) circle (0.8pt);
\filldraw [black] (3,0) node[anchor=west] {$x_p$};
\path (3.4,3) -- node[auto=false]{$\vdots$} (4.6,2.8);
\draw[black, thin] (5,2) -- (3.5,4.2) -- (6.5,4.2)-- cycle;
\filldraw [black] (5,2) circle (1pt);
\filldraw [black] (5,2) node[anchor=north] {$w$};
\path (7,2.5) -- node[auto=false]{$\vdots$} (5,3.2);
\draw[black, thin] (5,2) -- (8,3) -- (7,0)-- cycle;
\draw[black, thin] (3,0) -- (3.2,-0.7) -- (3.5,-0.6)-- cycle;
\filldraw [black] (3.2,-0.7) circle (0.5pt);
\filldraw [black] (3.5,-0.6) circle (0.5pt);
\path (2.5,-0.5) -- node[auto=false]{$\dots$} (3.3,-0.5);
\draw[black, thick] (3,0) -- (2.4,-0.7) -- (2.2,-0.5)-- cycle;
\filldraw [black, thick] (2.4,-0.7) circle (0.5pt);
\filldraw [black] (2.4,-0.7) node[anchor=north] {{\tiny $y_{p,2t-1}$}};
\filldraw [black, thick] (2.2,-0.5) circle (0.5pt);
\filldraw [black] (2.1,-0.5) node[anchor=south] {{\tiny $y_{p,2t}$}};
\draw[black, thin] (3,0) -- (2.5,0.7) -- (2.2,0.5)-- cycle;
\path (2.7,0.2) -- node[auto=false]{$\vdots$} (2.3,0.1);
\filldraw [black] (2.5,0.7) circle (0.5pt);
\filldraw [black] (2.2,0.5) circle (0.5pt);
\draw[black, thin] (2,3) -- (1.2,2.3) -- (1.4,2)-- cycle;
\filldraw [black] (1.2,2.3) circle (0.5pt);
\filldraw [black] (1.4,2) circle (0.5pt);
\draw[black, thin] (2,3) -- (1.2,3.3) -- (1.4,3.7)-- cycle;
\filldraw [black] (1.2,3.3) circle (0.5pt);
\filldraw [black] (1.4,3.7) circle (0.5pt);
\path (1,3) -- node[auto=false]{$\vdots$} (2,3);
\draw[black, thin] (3.5,4.2) -- (2.6,4) -- (3,3.52)-- cycle;
\filldraw [black] (2.6,4) circle (0.5pt);
\filldraw [black] (3,3.52) circle (0.5pt);
\draw[black, thin] (3.5,4.2) -- (3.2,4.8) -- (3.8,5)-- cycle;
\filldraw [black] (3.2,4.8) circle (0.5pt);
\filldraw [black] (3.8,5) circle (0.5pt);
\path (3,4.5) -- node[auto=false]{$\vdots$} (3.5,4.5);
\draw[black, thin] (6.5,4.2) -- (6.6,3.4) -- (7,3.52)-- cycle;
\filldraw [black] (6.6,3.4) circle (0.5pt);
\filldraw [black] (7,3.52) circle (0.5pt);
\draw[black, thick] (6.5,4.2) -- (6.8,4.8) -- (7.5,4.8)-- cycle;
\filldraw [black] (6.3,4.3) node[anchor=north] {$x_{q}$};
\filldraw [black, thick] (6.8,4.8) circle (0.5pt);
\filldraw [black] (6.8,4.7) node[anchor=south] {{\tiny $y_{q,2t^{\prime}-1}$}};
\filldraw [black, thick] (7.5,4.8) circle (0.5pt);
\filldraw [black] (7.5,4.8) node[anchor=west] {{\tiny $y_{q,2t^{\prime}}$}};
\path (6.7,4) -- node[auto=false]{$\vdots$} (7,4.5);

\draw[black, thin] (6.5,4.2) -- (5.3,4.5) -- (5.6,5)-- cycle;
\filldraw [black] (5.3,4.5) circle (0.5pt);
\filldraw [black] (5.6,5) circle (0.5pt);
\path (6.5,4.2) -- node[auto=false]{$\dots$} (6.4,5);
\draw[black, thin] (8,3) -- (7.8,3.4) -- (8.3,3.5)-- cycle;
\draw[black, thin] (8,3) -- (8.1,2.6) -- (8.5,2.4)-- cycle;
\filldraw [black] (8.1,2.6) circle (0.5pt);
\filldraw [black] (8.5,2.4) circle (0.5pt);
\path (8,2.7) -- node[auto=false]{$\vdots$} (8.5,3.7);
\filldraw [black] (8,3) circle (0.8pt);
\filldraw [black] (7.8,3.4) circle (0.5pt);
\filldraw [black] (8.3,3.5) circle (0.5pt);
\draw[black, thin] (7,0) -- (6,-0.5) -- (6.5,-0.5)-- cycle;
\filldraw [black] (6,-0.5) circle (0.5pt);
\filldraw [black] (6.5,-0.5) circle (0.5pt);
\draw[black, thin] (7,0) -- (7.34,-0.5) -- (7.6,-0.3)-- cycle;
\filldraw [black] (7.34,-0.5) circle (0.5pt);
\filldraw [black] (7.6,-0.3) circle (0.5pt);
\path (6.36,-0.6) -- node[auto=false]{$\dots$} (7.6,0);
\end{tikzpicture}
\caption{Illustration of an exceptional pair}
\label{fig:cactexpl}
\end{figure}
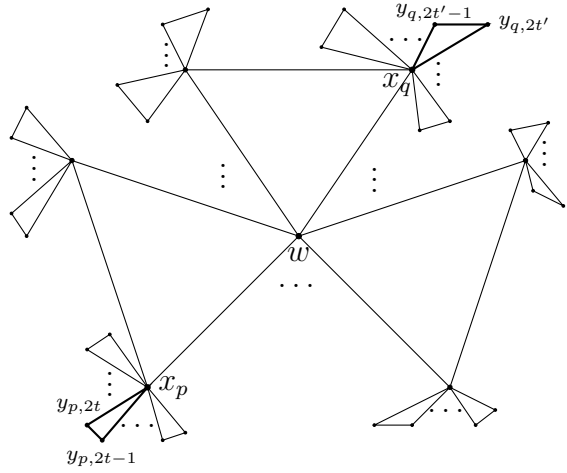

For any odd cycle $C$ in $\mathbf{G}$, recall that $\mathbb{E}_{C} \coloneqq \sum_{i\in V(C)}\mathbf{e}_{i},$ where $\mathbf{e}_{i}$ represents the $i$-th canonical vector of $\mathbb{R}^{d}$. 

\begin{lemma}\label{lemm:cact01}
Let $(C_{i},C_{i^{\prime}})$ and $(C_{j},C_{j^{\prime}})$ be any two exceptional pairs in $\mathbf{G}$. Then $$\mathbb{E}_{C_i}+\mathbb{E}_{C_{i^{\prime}}}+\mathbb{E}_{C_{j}}+\mathbb{E}_{C_{j^{\prime}}}\in S_{\mathbf{G}}$$ if and only if both $(C_{i},C_{j})$ and $(C_{i^{\prime}},C_{j^{\prime}})$ $\Big( \text{ or }(C_{i},C_{j^{\prime}})$ and $(C_{i^{\prime}},C_{j})\Big)$ are not exceptional.
\end{lemma}

\begin{proof}
Let us consider $C_{p}=\{x_{p},y_{p,2t_{p}-1},y_{p,2t_{p}}\}$, where $p\in\{i,i^{\prime},j,j^{\prime}\}$ and for each $p$, $t_{p}\in [s_{p}]$. 
We know that the pair $(C_{p},C_{q})$ is exceptional in $\mathbf{G}$ only if both $x_{p}\neq x_{q}$ and $\{x_{p}, x_{q}\}\notin E(\mathbf{G})$.
\medskip

($\implies$)
Let us assume that $\mathbb{E}_{C_i}+\mathbb{E}_{C_{i^{\prime}}}+\mathbb{E}_{C_{j}}+\mathbb{E}_{C_{j^{\prime}}}\in S_{\mathbf{G}}$.
This implies that $\mathbb{E}_{C_i}+\mathbb{E}_{C_{i^{\prime}}}+\mathbb{E}_{C_{j}}+\mathbb{E}_{C_{j^{\prime}}}$ is a linear combination of $\rho(e)$ for some $e\in E(\mathbf{G})$.

\noindent
According to the structure of the graph $\mathbf{G}$, $\{y_{u,k},y_{v,k^{\prime}}\}\notin E(\mathbf{G})$, for any distinct $u, v\in [2n]$, $k\in[2s_{u}]$, and $k^{\prime}\in[2s_{v}]$. 
Moreover, for the exceptional pairs $(C_{i},C_{i^{\prime}})$ and $(C_{j},C_{j^{\prime}})$, we know that $x_{i}\neq x_{i^{\prime}}$, $x_{j}\neq x_{j^{\prime}}$, $\{x_{i}, x_{i^{\prime}}\}\notin E(\mathbf{G})$, and $\{x_{j}, x_{j^{\prime}}\}\notin E(\mathbf{G})$. 
Therefore, $\mathbb{E}_{C_i}+\mathbb{E}_{C_{i^{\prime}}}+\mathbb{E}_{C_{j}}+\mathbb{E}_{C_{j^{\prime}}}$ is a $\mathbb{Z}_{\geq 0}$-linear combination of $\rho(e)$ for some $e\in E(\mathbf{G})$ only if both of the vertices $x_{i}$ and $x_{i^{\prime}}$ are such that 
\[
x_{i},x_{i^{\prime}}\in \{ x_{j},x_{j^{\prime}}\}\cup N_{\mathbf{G}}\big(\{ x_{j},x_{j^{\prime}}\}\big).
\]
This also implies $ x_{j},x_{j^{\prime}}\in \{ x_{i},x_{i^{\prime}}\}\cup N_{\mathbf{G}}\big(\{ x_{i},x_{i^{\prime}}\}\big).$
That is, either both $\mathbb{E}_{C_i}+\mathbb{E}_{C_{j}}\in S_{\mathbf{G}}$ and $\mathbb{E}_{C_{i^{\prime}}}+\mathbb{E}_{C_{j^{\prime}}}\in S_{\mathbf{G}}$,
or both $\mathbb{E}_{C_i}+\mathbb{E}_{C_{j^{\prime}}}\in S_{\mathbf{G}}$ and $\mathbb{E}_{C_{i^{\prime}}}+\mathbb{E}_{C_{j}}\in S_{\mathbf{G}}$.
Therefore, in the given odd cycles $C_{p}$ ($p= i,i^{\prime},j,j^{\prime}$), there exist two pairs of odd cycles that are not exceptional.
\medskip

($\impliedby$) Let us assume that both the pairs $(C_{i},C_{j})$ and  $(C_{i^{\prime}},C_{j^{\prime}})$, are not exceptional. 
Therefore, we have one of the following cases:
\begin{itemize}
    \item $V(C_{i})\cap V(C_{j})\neq \emptyset$ and $V(C_{i^{\prime}})\cap V(C_{j^{\prime}})\neq \emptyset$; 
    \item there exists a bridge between $C_{i},C_{j}$ and between $C_{i^{\prime}},C_{j^{\prime}}$;
    \item $V(C_{i})\cap V(C_{j})\neq \emptyset$ and  there exists a bridge between $C_{i^{\prime}},C_{j^{\prime}}$ or vice versa.
\end{itemize}

\noindent
In any of these cases, the method of the proof remains the same and therefore, without loss of generality, we assume that $C_{i}, C_{j}$ shares a common vertex, i.e., $x_{i}=x_{j}$ and there exists a bridge between $C_{i^{\prime}}, C_{j^{\prime}}$, i.e., $\{x_{i^{\prime}},x_{j^{\prime}}\}\in E(\mathbf{G})$. Hence,
\medskip

$\begin{array} {lcl}  \mathbb{E}_{C_i}+\mathbb{E}_{C_{i^{\prime}}}+\mathbb{E}_{C_{j}}+\mathbb{E}_{C_{j^{\prime}}}& = & (\mathbf{e}_{x_{i}}+\mathbf{e}_{y_{i,2t_{i}}}) + (\mathbf{e}_{x_{i}}+\mathbf{e}_{y_{i,2t_{i}-1}})\\
& &+ (\mathbf{e}_{y_{j,2t_{j}}}+\mathbf{e}_{y_{j,2t_{j}-1}}) + (\mathbf{e}_{x_{i^{\prime}}}+\mathbf{e}_{x_{j^{\prime}}}) \\
& & +  (\mathbf{e}_{y_{i^{\prime},2t_{i^{\prime}}}}+\mathbf{e}_{y_{i^{\prime},2t_{i^{\prime}}-1}}) + (\mathbf{e}_{y_{j^{\prime},2t_{j^{\prime}}}}+\mathbf{e}_{y_{j^{\prime},2t_{j^{\prime}}-1}})  .
\end{array}$ 
\medskip

\noindent
This implies, $\mathbb{E}_{C_i}+\mathbb{E}_{C_{i^{\prime}}}+\mathbb{E}_{C_{j}}+\mathbb{E}_{C_{j^{\prime}}}$ is a linear combination of $\rho(e)$ for some $e\in E(\mathbf{G})$. Thus,
$\mathbb{E}_{C_i}+\mathbb{E}_{C_{i^{\prime}}}+\mathbb{E}_{C_{j}}+\mathbb{E}_{C_{j^{\prime}}}\in S_{\mathbf{G}}$.
\end{proof}

\begin{lemma}\label{lemm:cact02}
Let $(C_{i},C_{i^{\prime}})$ be an exceptional pair in $\mathbf{G}$. For any edge $\{u,v\}\in E(\mathbf{G})$, we have that
$$\mathbb{E}_{C_{i}}+\mathbb{E}_{C_{i^{\prime}}}+\mathbf{e}_{u}+\mathbf{e}_{v}\in S_{\mathbf{G}}$$
if and only if one of the vertices $u$ or $v$ is the vertex $w$, while the other vertex belongs to the set $V(C_{i})\cup V(C_{i^{\prime}})\cup N_{\mathbf{G}}\big(V(C_{i})\cup V(C_{i^{\prime}})\big)$.
\end{lemma}

\begin{proof}
Let $C_{i}=\{x_{i},y_{i,2t_{i}-1},y_{i,2t_{i}}\}$ and $C_{i^{\prime}}=\{x_{i^{\prime}},y_{i^{\prime},2t_{i^{\prime}}-1},y_{i^{\prime},2t_{i^{\prime}}}\}$, for some $i,i^{\prime}\in [2n]$ with $t_{i}\in [s_{i}]$ and $t_{i^{\prime}}\in [s_{i^{\prime}}]$. 
We know that the pair $(C_{i},C_{i^{\prime}})$ is exceptional in $\mathbf{G}$ only if both
$V(C_{i})\cap V(C_{i^{\prime}})=\emptyset$ and  
$N_{\mathbf{G}}(V(C_{i}))\cap N_{\mathbf{G}}(V(C_{i^{\prime}}))=\{w\}$ (see Figure~\ref{fig:cactexpl}).
\medskip

($\implies$) Let us assume that for some edge $\{u,v\}\in E(\mathbf{G})$, 
\begin{equation}\label{eq:cact02}
\mathbb{E}_{C_{i}}+\mathbb{E}_{C_{i^{\prime}}}+\mathbf{e}_{u}+\mathbf{e}_{v}\in S_{\mathbf{G}}.   
\end{equation}
Since $(C_{i},C_{i^{\prime}})$ is exceptional in $\mathbf{G}$, we have $(\mathbb{E}_{C_{i}}+\mathbb{E}_{C_{i^{\prime}}})\notin S_{\mathbf{G}}$.
Let $(\mathbb{E}_{C})_{j}$ denote the $j$-th coordinate of $\mathbb{E}_{C}=\sum_{i\in V(C)}\mathbf{e}_{i},$ for any odd cycle $C$ in $\mathbf{G}$. It follows that $\sum_{j=1}^{d}(\mathbb{E}_{C})_{j}=3$, for any odd cycle $C$ in $\mathbf{G}$. 
Therefore for \eqref{eq:cact02} to be true, we require either
\begin{itemize}
    \item $\mathbb{E}_{C_{i}}+\mathbf{e}_{u}\in S_{\mathbf{G}}$ and $\mathbb{E}_{C_{i^{\prime}}}+\mathbf{e}_{v}\in S_{\mathbf{G}}$; or 
    \item $\mathbb{E}_{C_{i}}+\mathbf{e}_{v}\in S_{\mathbf{G}}$ and $\mathbb{E}_{C_{i^{\prime}}}+\mathbf{e}_{u}\in S_{\mathbf{G}}$.
\end{itemize}
In any of these cases, the proof method remains consistent. Therefore, without loss of generality, we can assume that
$\mathbb{E}_{C_{i}}+\mathbf{e}_{u}\in S_{\mathbf{G}}$ and
$\mathbb{E}_{C_{i^{\prime}}}+\mathbf{e}_{v}\in S_{\mathbf{G}}$.
This implies, $u\in V(C_{i})\cup N_{\mathbf{G}}\big(V(C_{i})\big)$ and $v\in V(C_{i^\prime})\cup N_{\mathbf{G}}\big(V(C_{i^{\prime}})\big)$.
For clarity, we will denote $x_{\hat{i}}$ as the vertex adjacent to $x_{i}$ that belongs to the $3$-cycle containing $w$.
In other words, $\{x_{i},x_{\hat{i}}\}\in E(\mathbf{G})$ and both $x_{\hat{i}}$ and $x_{i}$ are part of a $3$-cycle that includes the vertex $w$.
Similarly, we assume $\{x_{i^\prime},x_{\hat{i}^\prime}\}\in E(\mathbf{G})$.
Since $\{u,v\}\in E(\mathbf{G})$ and $(C_{i},C_{i^{\prime}})$ is exceptional in $\mathbf{G}$, we note that the vertex $u$ can not be $y_{i,2t_{i}-1}$ or $y_{i,2t_{i}}$. Similarly, $v$ is not any of the vertices $y_{i^{\prime},2t_{i^{\prime}}-1}$ or $y_{i^{\prime},2t_{i^{\prime}}}$.
Thus, we conclude
$u\in \{x_{i}\} \cup \{w,x_{\hat{i}}\}$ and $v\in \{x_{i^\prime}\} \cup \{w, x_{\hat{i}^\prime}\}$.
Suppose $u=x_{i}$. Since $\mathbf{G}$ is a triangular cactus graph with $\diam{\mathbf{G}}=4$, we have that $\{x_{i}, x_{\hat{i}^{\prime}}\}\notin E(\mathbf{G})$, and since $(C_{i},C_{i^{\prime}})$ is an exceptional pair, $\{x_{i}, x_{i^{\prime}}\}\notin E(\mathbf{G})$. Therefore, the only possibility for vertex $v$ is that $v=w$. Similarly, if $u=x_{\hat{i}}$, then also $v$ must be $w$.
If $u=w$, then it is clear that $v$ is either $x_{i^\prime}$ or $x_{\hat{i}^\prime}$.

\noindent
Hence if \eqref{eq:cact02} holds, then one of the vertices $u$ or $v$ must be $w$, while the other vertex belongs to the set $V(C_{i})\cup V(C_{i^{\prime}})\cup N_{\mathbf{G}}\big(V(C_{i})\cup V(C_{i^{\prime}})\big)$.
\medskip

($\impliedby$) Let $\{u,v\}=\{v,w\}\in E(\mathbf{G})$ such that $v\in V(C_{i})\cup V(C_{i^{\prime}})\cup N_{\mathbf{G}}\big(V(C_{i})\cup V(C_{i^{\prime}})\big)$.
Since $N_{\mathbf{G}}(V(C_{i}))\cap N_{\mathbf{G}}(V(C_{i^{\prime}}))=\{w\}$, both $\mathbb{E}_{C_i}+\mathbf{e}_{w}$ and $\mathbb{E}_{C_{i^{\prime}}}+\mathbf{e}_{w}$ can be expressed as a linear combination of $\rho(e)$ for some $e\in E(\mathbf{G})$.
Now, without loss of generality, we assume that $v\in V(C_{i})\cup N_{\mathbf{G}}\big(V(C_{i})\big)$. 
Thus we can write $\mathbb{E}_{C_{i}}+\mathbf{e}_{v}$ as a linear combination of $\rho(e)$ for some $e\in E(\mathbf{G})$.
Hence, $\mathbb{E}_{C_{i}}+\mathbb{E}_{C_{i^{\prime}}}+\mathbf{e}_{v}+\mathbf{e}_{w} =  (\mathbb{E}_{C_{i}}+\mathbf{e}_{v})+(\mathbb{E}_{C_{i^{\prime}}}+\mathbf{e}_{w})$, can be expressed as a linear combination of $\rho(e)$ for some $e\in E(\mathbf{G})$.
Therefore, we have $\mathbb{E}_{C_{i}}+\mathbb{E}_{C_{i^{\prime}}}+\mathbf{e}_{v}+\mathbf{e}_{w}\in S_{\mathbf{G}}$.
\end{proof}

\begin{lemma}\label{lemm:cact03}
Let $(C,C^{\prime})$ be an exceptional pair in $\mathbf{G}$. For any edges $\{u,w\}$ and $\{v,w\}$ in $E(\mathbf{G})$ such that both $u$ and $v$ are not in $V(C)\cup V(C^{\prime})\cup N_{\mathbf{G}}\big(V(C)\cup V(C^{\prime})\big)$, we have
$$\mathbb{E}_{C}+\mathbb{E}_{C^{\prime}}+\rho(\{u,w\})+\rho(\{v,w\})\in S_{\mathbf{G}}$$
if and only if $\{u,v\}\in E(\mathbf{G})$.
\end{lemma}

\begin{proof}
Since $\{u,w\}$ and $\{v,w\}$ are edges of $\mathbf{G}$, observe that $u=x_{i}$ and $v=x_{j}$ for some $i,j\in [2n]$ (see Figure~\ref{fig:cactgen}). 
Recall that for any exceptional pair $(C,C^{\prime})$ in $\mathbf{G}$, we have
$(\mathbb{E}_{C}+\mathbb{E}_{C{\prime}})\notin S_{\mathbf{G}}$.
Furthermore, since both $u$ and $v$ are not in $V(C)\cup V(C^{\prime})\cup N_{\mathbf{G}}\big(V(C)\cup V(C^{\prime})\big)$, it follows that none of the expressions 
$(\mathbb{E}_{C}+ \mathbf{e}_{u})$, $(\mathbb{E}_{C}+\mathbf{e}_{v})$, $ (\mathbb{E}_{C^{\prime}}+ \mathbf{e}_{u})$, or $ (\mathbb{E}_{C^{\prime}}+ \mathbf{e}_{v})$ can be expressed as a linear combination of $\rho(e)$ for some $e\in E(\mathbf{G})$. 
Thus, for $$\mathbb{E}_{C}+\mathbb{E}_{C^{\prime}}+\rho(\{u,w\})+\rho(\{v,w\})= \mathbb{E}_{C}+\mathbb{E}_{C^{\prime}} + \mathbf{e}_{u}+ \mathbf{e}_{w} + \mathbf{e}_{v} + \mathbf{e}_{w} \in S_{\mathbf{G}},$$ the only possible combination is 
$(\mathbb{E}_{C}+ \mathbf{e}_{w}) + (\mathbb{E}_{C^{\prime}}+\mathbf{e}_{w}) + (\mathbf{e}_{u}+\mathbf{e}_{v}).$ 
For this expression to belong to $S_{\mathbf{G}}$, it is necessary that $(\mathbf{e}_{u}+\mathbf{e}_{v})\in S_{\mathbf{G}}$.  
This is possible if and only if $\{u,v\}\in E(\mathbf{G})$.
\end{proof}

Note that $w$ is always a cutpoint of $\mathbf{G}$.
Moreover, all the non-cutpoints of $\mathbf{G}$ are regular in $\mathbf{G}$.
Based on the fact on whether $w$ in $\mathbf{G}$ is a regular cutpoint or not, we classify the graph $\mathbf{G}$ into two types.
Let us look at these two cases in-depth in the following subsections.


\subsection{\texorpdfstring{Type 1: $w$ is a regular cutpoint of $\mathbf{G}$}{Type 1}}\label{subsec:c1}

Let $\mathbf{G}^{\prime}$ be the triangular cactus graph $\mathbf{G}$ where any $3$-cycle in $\mathbf{G}^{\prime}$ containing the vertex $w$ is such that at least one of its remaining vertices $x_{i}$ will have at least one $3$-cycle of the form $\{x_{i},y_{i,2t-1},y_{i,2t}\}$ attached to it. 
This is the first type of triangular cactus graph $\mathbf{G}$ that we will be studying.

Let us denote the vertices $x_i$ with no $3$-cycles $\{x_{i},y_{i,2t-1},y_{i,2t}\}$ attached to them as $\zeta_i$ and let there be $l$ such vertices $\zeta_i$ in $\mathbf{G}^{\prime}$.
That is, each $\zeta_i\in V(\mathbf{G}^{\prime})$ has no $3$-cycles $\{x_{i},y_{i,2t-1},y_{i,2t}\}$ attached to it for all $1\leq i\leq l$.
Note that, according to the description of $\mathbf{G}^{\prime}$, for any edge $\{x_{i},x_{j}\}\in E(\mathbf{G}^{\prime})$, at most one of the vertices ($x_{i}$ or $x_{j}$) can be a vertex $\zeta_i$.
Therefore, we have $0\leq l\leq n$.
An illustration of the graph $\mathbf{G}^{\prime}$ is shown in Figure~\ref{fig:cactreg}.

\begin{figure}[ht]
\centering
\begin{tikzpicture}
\draw[black, thin] (3,0) -- (5,2) -- (2,3)-- cycle;
\filldraw [black] (3,0) circle (0.8pt);
\filldraw [black] (2,3) circle (0.8pt);
\filldraw [black] (3.5,4.2) circle (0.8pt);
\filldraw [black] (6.5,4.2) circle (0.8pt);
\filldraw [black] (8,3) circle (0.8pt);
\filldraw [black] (7,0) circle (0.8pt);
\filldraw [black] (3,0) node[anchor=west] {$\zeta_i$};
\filldraw [black] (2,3) circle (0.8pt);
\filldraw [black] (2,3.1) node[anchor=west] {$x_j$};
\draw[black, thin] (5,2) -- (3.5,4.2) -- (6.5,4.2)-- cycle;
\filldraw [black] (5,2) circle (1pt);
\filldraw [black] (5,2) node[anchor=north] {$w$};
\path (7,2.5) -- node[auto=false]{$\vdots$} (5,3.2);
\draw[black, thin] (5,2) -- (8,3) -- (7,0)-- cycle;
\draw[black, thin] (2,3) -- (1.2,2.3) -- (1.4,2)-- cycle;
\filldraw [black] (1.2,2.3) circle (0.5pt);
\filldraw [black] (1.2,2.3) node[anchor=east] {{\tiny $y_{j,2}$}};
\filldraw [black] (1.4,2) circle (0.5pt);
\filldraw [black] (1.4,2) node[anchor=north] {{\tiny $y_{j,1}$}};
\draw[black, thin] (2,3) -- (1.2,3.3) -- (1.4,3.7)-- cycle;
\filldraw [black] (1.2,3.3) circle (0.5pt);
\filldraw [black] (1.3,3.3) node[anchor=east] {{\tiny $y_{j,2s_{j}-1}$}};
\filldraw [black] (1.4,3.7) circle (0.5pt);
\filldraw [black] (1.4,3.7) node[anchor=south] {{\tiny $y_{j,2s_{j}}$}};
\path (1,3) -- node[auto=false]{$\vdots$} (2,3);
\draw[black, thin] (3.5,4.2) -- (2.6,4) -- (3,3.52)-- cycle;
\filldraw [black] (2.6,4) circle (0.5pt);
\filldraw [black] (3,3.52) circle (0.5pt);
\draw[black, thin] (3.5,4.2) -- (3.2,4.8) -- (3.8,5)-- cycle;
\filldraw [black] (3.2,4.8) circle (0.5pt);
\filldraw [black] (3.8,5) circle (0.5pt);
\path (3,4.5) -- node[auto=false]{$\vdots$} (3.5,4.5);
\draw[black, thin] (6.5,4.2) -- (6.6,3.4) -- (7,3.52)-- cycle;
\filldraw [black] (6.6,3.4) circle (0.5pt);
\filldraw [black] (7,3.52) circle (0.5pt);
\draw[black, thin] (6.5,4.2) -- (6.8,4.8) -- (7.5,4.8)-- cycle;
\filldraw [black] (6.8,4.8) circle (0.5pt);
\filldraw [black] (7.5,4.8) circle (0.5pt);
\path (6.7,4) -- node[auto=false]{$\vdots$} (7,4.5);
\draw[black, thin] (8,3) -- (7.8,3.4) -- (8.3,3.5)-- cycle;
\filldraw [black] (7.8,3.4) circle (0.5pt);
\filldraw [black] (8.3,3.5) circle (0.5pt);
\draw[black, thin] (8,3) -- (8.1,2.6) -- (8.5,2.4)-- cycle;
\filldraw [black] (8.1,2.6) circle (0.5pt);
\filldraw [black] (8.5,2.4) circle (0.5pt);
\path (8,2.7) -- node[auto=false]{$\vdots$} (8.5,3.7);
\draw[black, thin] (7,0) -- (6,-0.5) -- (6.5,-0.5)-- cycle;
\filldraw [black] (6,-0.5) circle (0.5pt);
\filldraw [black] (6.5,-0.5) circle (0.5pt);
\draw[black, thin] (7,0) -- (7.34,-0.5) -- (7.6,-0.3)-- cycle;
\filldraw [black] (7.34,-0.5) circle (0.5pt);
\filldraw [black] (7.6,-0.3) circle (0.5pt);
\path (6.36,-0.6) -- node[auto=false]{$\dots$} (7.6,0);
\end{tikzpicture}
\caption{An illustration of graph $\mathbf{G}^{\prime}$}
\label{fig:cactreg}
\end{figure}
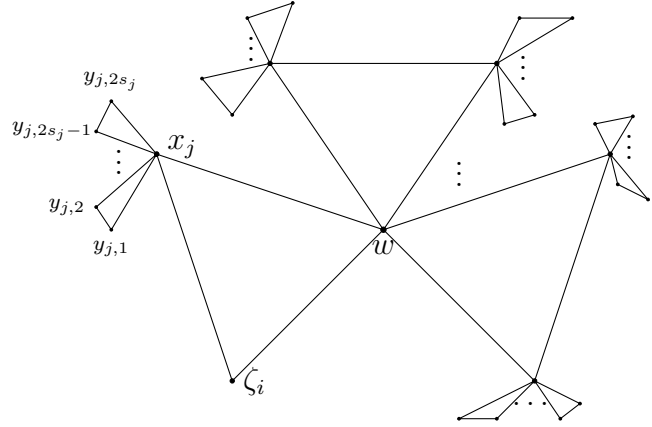

Now, let us look at the regular vertices and fundamental sets in $\mathbf{G}^{\prime}$. 
The vertex $w$ and all the $x_{i}$ of $\mathbf{G}^{\prime}$ with at least one $3$-cycle $\{x_{i},y_{i,2t-1},y_{i,2t}\}$ attached to them, are cutpoints of $\mathbf{G}^{\prime}$. That is, the set of cutpoints in $\mathbf{G}^{\prime}$ is 
$\{w\}\cup \big(\{x_{i}\colon 1\leq i\leq 2n\}\backslash \bigcup_{i=1}^{l}\{\zeta_{i}\}\big) .$
The vertex $w$ and all the non-cutpoints of $\mathbf{G}^{\prime}$ are regular in $\mathbf{G}^{\prime}$. 

Let us now focus on the fundamental sets in $\mathbf{G}^{\prime}$.
For $i\in [2n]$, let $s_i$ number of $3$-cycles, not containing the vertex $w$, be attached to $x_{i}\in V(\mathbf{G}^{\prime})$.
All the fundamental sets of $\mathbf{G}^{\prime}$ containing the vertex $w$ have the form $\{w\}\cup\bigcup_{1\leq i \leq 2n}\{y_{i,k_{1}},y_{i,k_{2}},\dots, y_{i,k_{s_i}}\}$, where $k_{u}\in [2s_i]$ and none of the vertices $y_{i,k_{u}}$ are adjacent to each other. 
Note that, for a fundamental set $T$ with $w\in T$, and any exceptional pair $(C, C^{\prime})$ in $\mathbf{G}^{\prime}$, 
$\big(T\cup N_{\mathbf{G}^{\prime}}(T)\big) \cap \big(V(C)\cup V(C^{\prime})\big) \neq \emptyset$.

For the graph $\mathbf{G}^{\prime}$, and for any $i,j\in [2n]$, the building blocks of any fundamental set not containing $w$ are as follows:
\begin{enumerate}[label=(\roman*)]
    \item \label{1} Suppose $s$ is the number of $3$-cycles that are attached to the vertex $x_{i}$ but do not contain the vertex $w$. Then, the set $\{y_{i,k_{1}}, y_{i,k_{2}}, \dots, y_{i,k_{s}}\}$, where $k_{u}\in [2s]$, and none of the vertices $y_{i,k_{u}}$ are adjacent to each other, is fundamental in $\mathbf{G}^{\prime}$.
    \item \label{2} Let $\{x_{i}, x_{j}\} \in E(\mathbf{G}^{\prime})$, and let $t$ be the number of $3$-cycles attached to the vertex $x_{j}$ that do not contain the vertex $w$. Then, the set $\{x_{i}, y_{j,k_{1}}, y_{j,k_{2}}, \dots, y_{j,k_{t}}\}$, where none of the vertices $y_{j,k_{v}}$ are adjacent to each other for $k_{v}\in [2t]$, forms a fundamental set in $\mathbf{G}^{\prime}$.
\end{enumerate}

Note that any fundamental set $T$ in $\mathbf{G}^{\prime}$ with $w\notin T$ is either one of the listed building blocks or can be expressed as their unions.
Moreover, for any fundamental set $T$ in $\mathbf{G}^{\prime}$ consisting of building blocks listed in \ref{2}, we have $w\in N_{\mathbf{G}^{\prime}}(T)$.

Let us define some notations to ensure clarity and consistency throughout this article. 
Suppose $p$ is a positive integer such that $1\leq p\leq \floor*{\frac{n}{2}}$, where $n$ is the total number of $3$-cycles with common vertex $w$.
Let $\mathcal{E}_{\mathbf{G}^{\prime}}^{p}[(C, C^{\prime})]$ denote the set consisting of $p$ exceptional pairs of $\mathbf{G}^{\prime}$, including the exceptional pair $(C, C^{\prime})$, such that pairing any odd cycles among the exceptional pairs in $\mathcal{E}_{\mathbf{G}^{\prime}}^{p}[(C, C^{\prime})]$ yields another exceptional pair in $\mathbf{G}^{\prime}$. Specifically, we have
\begin{align*}
\mathcal{E}_{\mathbf{G}^{\prime}}^{p}[(C, C^{\prime})]\\ 
& \hspace{-2cm}  =
\left\{(C_1, C_{1^\prime}),\dots, (C_p, C_{p^\prime})\colon 
\begin{aligned}
\phantom{a} & (C_i, C_{i^\prime}) \textrm{ is exceptional in } \mathbf{G}^{\prime},\  \forall \ 1\leq i\leq p,\\
\phantom{b} & (C, C^{\prime}) =(C_{k}, C_{k^\prime}) \textrm{ for some } k\in[p], \textrm{ and }\\
\phantom{c} & (C_i, C_{j}) \textrm{ is also exceptional in } \mathbf{G}^{\prime}  \textrm{ for any } i,j\in [p\cup p^{\prime}].
\end{aligned}
\right\}
\end{align*}
Note that the collection of all exceptional pairs of $\mathbf{G}^{\prime}$ can be denoted as the union of the sets $\mathcal{E}_{\mathbf{G}^{\prime}}^{1}[(C, C^{\prime})]=\{(C, C^{\prime})\}$, taken over all possible exceptional pairs $(C, C^{\prime})$ of $\mathbf{G}^{\prime}$.

For any set $\mathcal{E}_{\mathbf{G}^{\prime}}^{p}[(C, C^{\prime})]$, we define 
$$q_{\mathcal{E}_{p}}[(C, C^{\prime})]\coloneqq \sum\limits_{i=1}^{p}(\mathbb{E}_{C_{i}}+ \mathbb{E}_{C_{i^{\prime}}}),$$
where $(C_i, C_{i^\prime}) \in \mathcal{E}_{\mathbf{G}^{\prime}}^{p}[(C, C^{\prime})]$ for all $1\leq i\leq p$.
Since pairing any odd cycles from the exceptional pairs in $\mathcal{E}_{\mathbf{G}^{\prime}}^{p}[(C, C^{\prime})]$ results in another exceptional pair in $\mathbf{G}^{\prime}$, in accordance with Lemma~\ref{lemm:cact01}, it follows that $q_{\mathcal{E}_{p}}[(C, C^{\prime})]\notin S_{\mathbf{G}^{\prime}}$, for any $\mathcal{E}_{\mathbf{G}^{\prime}}^{p}[(C, C^{\prime})]$.

Further, we define $\mathcal{T}_{\mathcal{E}_{\mathbf{G}^{\prime}}^{p}}[(C, C^{\prime})]$ as the set of all fundamental sets in $\mathbf{G}^{\prime}$ such that we have $w\in N_{\mathbf{G}^{\prime}}(T)$ and $\big(T\cup N_{\mathbf{G}^{\prime}}(T)\big) \cap \big(V(C_i)\cup V(C_{i^{\prime}})\big) =\emptyset$ for any exceptional pair $(C_i, C_{i^{\prime}}) \in \mathcal{E}_{\mathbf{G}^{\prime}}^{p}[(C, C^{\prime})]$. 
We may drop the notation $[(C, C^{\prime})]$ when the context is clear, referring simply to $\mathcal{T}_{\mathcal{E}_{\mathbf{G}^{\prime}}^{p}}$.

\begin{prop}\label{prop:hole1}
Let $(C, C^{\prime})$ be any exceptional pair in $\mathbf{G}^{\prime}$ and let $n$ be the total number of $3$-cycles in $\mathbf{G}^{\prime}$ with common vertex $w$.
For the triangular cactus graph $\mathbf{G}^{\prime}$ (Figure~\ref{fig:cactreg}), the set of holes $\overline{S_{\mathbf{G}^{\prime}}}\backslash S_{\mathbf{G}^{\prime}}$ is as follows:
\begin{equation}\label{eq:hole_cactus01}
\overline{S_{\mathbf{G}^{\prime}}}\backslash S_{\mathbf{G}^{\prime}} = \left( \bigcup\limits_{\substack {\mathcal{E}_{\mathbf{G}^{\prime}}^{p}[(C, C^{\prime})]; \\ \\ T\in\mathcal{T}_{\mathcal{E}_{\mathbf{G}^{\prime}}^{p}}}} (q_{\mathcal{E}_{p}}[(C, C^{\prime})]+\mathbb{Z}F_{T} ) \cap \mathcal{C}_{\mathbf{G}^{\prime}}\right) \cup \left(\bigcup\limits_{\mathcal{E}_{\mathbf{G}^{\prime}}^{p}[(C, C^{\prime})]} (q_{\mathcal{E}_{p}}[(C, C^{\prime})]+\mathbb{Z}F_{w}) \cap \mathcal{C}_{\mathbf{G}^{\prime}}\right), 
\end{equation}
where the union is taken over all possible sets $\mathcal{E}_{\mathbf{G}^{\prime}}^{p}[(C, C^{\prime})]$ for which $q_{\mathcal{E}_{p}}[(C, C^{\prime})]$ is defined in $\mathbf{G}^{\prime}$ and their corresponding fundamental sets $T\in\mathcal{T}_{\mathcal{E}_{\mathbf{G}^{\prime}}^{p}}[(C, C^{\prime})]$, where $1\leq p\leq \floor*{\frac{n}{2}}$.
\end{prop}

\begin{proof}
We know that 
$\overline{S_{\mathbf{G}^{\prime}}}=S_{\mathbf{G}^{\prime}} + \mathbb{Z}_{\geq 0}\{\mathbb{E}_{C}+ \mathbb{E}_{C^{\prime}}\colon (C,C^{\prime}) \text{ is exceptional in } \mathbf{G}^{\prime}\}$.
Note that $q_{\mathcal{E}_{p}}[(C, C^{\prime})]$ can be written as some element in $\mathbb{Z}_{\geq 0}\{\mathbb{E}_{C}+ \mathbb{E}_{C^{\prime}}\colon (C,C^{\prime}) \text{ is exceptional in } \mathbf{G}^{\prime}\}$ that can not be represented as a linear combination of $\rho(e)$ for some $e\in E(\mathbf{G}^{\prime})$.

($\supset$) Let $\alpha\in \left(q_{\mathcal{E}_{p}}[(C, C^{\prime})]+\mathbb{Z}F_{w} \right) \cap \mathcal{C}_{\mathbf{G}^{\prime}}$ for some $q_{\mathcal{E}_{p}}[(C, C^{\prime})]$. 
Let $\alpha_{i}$, $(q_{\mathcal{E}_{p}})_{i}$ and $(\mathbb{Z}F_{w})_{i}$ represent the $i$-th coordinates of $\alpha$, $q_{\mathcal{E}_{p}}[(C, C^{\prime})]$ and $\mathbb{Z}F_{w}$ respectively. 
We know that $q_{\mathcal{E}_{p}}[(C, C^{\prime})]\notin S_{\mathbf{G}^{\prime}}$.
Since $w\notin V(C)\cup V(C^{\prime})$ for any exceptional pair $(C, C^{\prime})$ in $\mathbf{G}^{\prime}$, it follows that $(q_{\mathcal{E}_{p}})_{w}=0$.
Additionally, $(\mathbb{Z}F_{w})_{w}=0$.
This implies that both $q_{\mathcal{E}_{p}}[(C, C^{\prime})]$ and $\mathbb{Z}F_{w}$ have no contributions from any edges adjacent to the vertex $w$. 
Therefore, we conclude that $\alpha_{w}=0$, indicating that $\alpha$ has no contribution from edges of the form $\{u,w\}\in E(\mathbf{G}^{\prime})$.
Thus, by Lemma~\ref{lemm:cact02}, it is clear that $\alpha\notin S_{\mathbf{G}^{\prime}}$. Consequently, we have $\alpha\in \overline{S_{\mathbf{G}^{\prime}}}\backslash S_{\mathbf{G}^{\prime}}$.

\noindent
Now, consider $\alpha^\prime\in \left(q_{\mathcal{E}_{p}}[(C, C^{\prime})]+\mathbb{Z}F_{T} \right) \cap \mathcal{C}_{\mathbf{G}^{\prime}}$ for some $q_{\mathcal{E}_{p}}[(C, C^{\prime})]$ and $T\in\mathcal{T}_{\mathcal{E}_{\mathbf{G}^{\prime}}^{p}}[(C, C^{\prime})]$.
By definition, for any $T\in\mathcal{T}_{\mathcal{E}_{\mathbf{G}^{\prime}}^{p}}[(C, C^{\prime})]$, the condition $\big(T\cup N_{\mathbf{G}^{\prime}}(T)\big) \cap \big(V(C_i)\cup V(C_{i^{\prime}})\big) =\emptyset$ holds for any exceptional pairs $(C_i, C_{i^{\prime}}) \in \mathcal{E}_{\mathbf{G}^{\prime}}^{p}[(C, C^{\prime})]$ for all $1\leq p\leq \floor*{\frac{n}{2}}$ and exceptional pairs $(C, C^{\prime})$ in $\mathbf{G^\prime}$.
Further, since $w\in N_{\mathbf{G^\prime}}(T)$, any $T\in\mathcal{T}_{\mathcal{E}_{\mathbf{G}^{\prime}}^{p}}[(C, C^{\prime})]$ is constructed from building blocks of the form specified in \ref{2}. 
In this construction, there will not be any edges of the form $\{x_{i}, x_{j}\}\in E(\mathbf{G}^{\prime})$ such that both edges $\{x_{i}, w\}$ and $\{x_{j},w\}$ are included in $T$.
Thus, by Lemma~\ref{lemm:cact03}, any combination of edges from the facet $F_{T}$ combined with $q_{\mathcal{E}_{p}}[(C, C^{\prime})]$ will never belong to $S_{\mathbf{G}^{\prime}}$, for $T\in\mathcal{T}_{\mathcal{E}_{\mathbf{G}^{\prime}}^{p}}[(C, C^{\prime})]$, for any  $1\leq p\leq \floor*{\frac{n}{2}}$ and exceptional pairs $(C, C^{\prime})$ in $\mathbf{G^\prime}$.
Therefore, we conclude that $\alpha^\prime \in \overline{S_{\mathbf{G}^{\prime}}}\backslash S_{\mathbf{G}^{\prime}}$.

($\subset$) Now let us prove that 
\begin{equation*}
\overline{S_{\mathbf{G}^{\prime}}}\backslash S_{\mathbf{G}^{\prime}} \subset \left(\bigcup\limits_{\substack {\mathcal{E}_{\mathbf{G}^{\prime}}^{p}[(C, C^{\prime})]; \\ \\ T\in\mathcal{T}_{\mathcal{E}_{\mathbf{G}^{\prime}}^{p}}}} (q_{\mathcal{E}_{p}}[(C, C^{\prime})]+\mathbb{Z}F_{T} ) \cap \mathcal{C}_{\mathbf{G}^{\prime}}\right)\cup \left( \bigcup\limits_{\mathcal{E}_{\mathbf{G}^{\prime}}^{p}[(C, C^{\prime})]} (q_{\mathcal{E}_{p}}[(C, C^{\prime}])+\mathbb{Z}F_{w}) \cap \mathcal{C}_{\mathbf{G}^{\prime}}\right).
\end{equation*}

\noindent
Let us consider $\alpha\in\overline{S_{\mathbf{G}^{\prime}}}\backslash S_{\mathbf{G}^{\prime}}$. 
Then we can express $\alpha$ as $\alpha= \beta +\gamma$, where $\beta\in S_{\mathbf{G}^{\prime}}$ and $\gamma\in \mathbb{Z}_{\geq 0}\{\mathbb{E}_{C}+ \mathbb{E}_{C^{\prime}}\colon (C,C^{\prime}) \text{ are exceptional in } \mathbf{G}^{\prime}\}$ such that $\alpha$ can not be expressed as a linear combination of $\rho(e)$ for some $e\in E(\mathbf{G}^{\prime})$.
If $\gamma = 0$, then $\alpha\in S_{\mathbf{G}^{\prime}} $. So, let us consider the case where $\gamma \neq 0$.
Let $\alpha_{i}$, $\beta_{i}$ and $\gamma_{i}$ represent the $i$-th coordinates of $\alpha$, $\beta$ and $\gamma$ respectively. 
\medskip

\noindent
$\textbf{\underline{Case 1:}}$  Let $\alpha_{w}=0$.
Given that $\alpha_{w}=0$ and $\gamma \neq 0$ with $\gamma_{w}=0$, it follows that if $\beta\neq 0$, then $\beta_{w}= 0$.
This indicates that the formation of $\beta$ does not include any edges adjacent to the common vertex $w$, implying that $\beta$ is contained within the facet $F_{w}$. 
Furthermore, from the structure of $\mathbf{G}^{\prime}$, we can see that for any arbitrary exceptional pairs $(C,C^{\prime})$ and $(\overline{C},\overline{C^{\prime}})$, if we have $\mathbb{E}_{C}+\mathbb{E}_{C^{\prime}}+\mathbb{E}_{\overline{C}}+\mathbb{E}_{\overline{C^{\prime}}}\in S_{\mathbf{G}^{\prime}}$ (as established in Lemma~\ref{lemm:cact01}),  then this expression is always contained within the facet $F_w$. 
Therefore, these structures can be included in $\beta$.
We know that $q_{\mathcal{E}_{p}}[(C, C^{\prime})] = \sum_{i=1}^{p}(\mathbb{E}_{C_{i}}+ \mathbb{E}_{C_{i^{\prime}}})\notin S_{\mathbf{G}^{\prime}}$ for any $\mathcal{E}_{\mathbf{G}^{\prime}}^{p}[(C, C^{\prime})]$.
Additionally, we have seen that $q_{\mathcal{E}_{p}}[(C, C^{\prime})]+\mathbb{Z}F_{w}$ cannot be expressed as any linear combination of $\rho(e)$ for $e\in E(\mathbf{G}^{\prime})$, and the $w$-th coordinate of $q_{\mathcal{E}_{p}}[(C, C^{\prime})]+\mathbb{Z}F_{w}$ will be $0$ for any $(C, C^{\prime})$. Therefore, we can express $\alpha= \left(q_{\mathcal{E}_{p}}[(C, C^{\prime})]+\mathbb{Z}F_{w} \right) \cap \mathcal{C}_{\mathbf{G}^{\prime}}$, for some $q_{\mathcal{E}_{p}}[(C, C^{\prime})]$, where all $\beta$ contained in the facet $F_{w}$ will be considered in $\mathbb{Z}F_{w}$.
\medskip


\noindent
$\textbf{\underline{Case 2:}}$  Let $\alpha_{w}> 0$. This implies $\beta\neq 0$ and $\beta_{w}>0$, since $\gamma_{w}=0$ (as $w$ is not contained in the vertex set of any odd cycles appearing in any exceptional pairs).
Consequently, among the edges defining the vector $\beta$, there must be at least one edge adjacent to $w$, say $\{v, w\}$.
Examining the structure of the fundamental sets, the presence of the edge $\{v, w\}$ in the formation of $\beta$ indicates that $\beta$ belongs to some facets corresponding to the fundamental sets with building blocks of type \ref{2}.
Additionally, by generalizing Lemma~\ref{lemm:cact03}, we observe that 
$\sum_{i=1}^{p}(\mathbb{E}_{C_{i}}+ \mathbb{E}_{C_{i^{\prime}}}) + \sum_{i=1}^{p}(\mathbf{e}_{v_i}+\mathbf{e}_{w})$ does not belong to $ S_{\mathbf{G}^{\prime}}$, when $(C_{i}, C_{i^{\prime}})\in \mathcal{E}_{\mathbf{G}^{\prime}}^{p}[(C, C^{\prime})]$ for some $p$ and any exceptional pair $(C, C^{\prime})$.
This holds under the condition that $v_{j}\notin V(C_{i})\cup V(C_{i}^{\prime})\cup N_{\mathbf{G}^{\prime}}\big(V(C_{i})\cup V(C_{i}^{\prime})\big)$ for all $i,j\in [p]$.
Thus, we conclude that the edges $\{v,w\}$ adjacent to $w$ in defining the vector $\beta$ are contained in some facet $F_{T}$ corresponding to the fundamental set $T$ of $\mathbf{G}^{\prime}$ such that $v\in T$, $w\in N_{\mathbf{G}^{\prime}}(T)$,  and $\big(\{v\}\cup N_{\mathbf{G}^{\prime}}(\{v\})\big) \cap \big(V(C_{i})\cup V(C_{i}^{\prime})\big) =\emptyset$.
Therefore, we conclude that any $\alpha\in\overline{S_{\mathbf{G}^{\prime}}}\backslash S_{\mathbf{G}^{\prime}}$ with $\alpha_{w}> 0$ can be expressed as 
$\left(q_{\mathcal{E}_{p}}[(C, C^{\prime})]+\mathbb{Z}F_{T}\right) \cap \mathcal{C}_{\mathbf{G}^{\prime}}$ for some $q_{\mathcal{E}_{p}}[(C, C^{\prime})]$ and its corresponding fundamental set $T\in\mathcal{T}_{\mathcal{E}_{\mathbf{G}^{\prime}}^{p}}[(C, C^{\prime})]$.
\medskip

\noindent
In both cases, we get the desired containment.
\end{proof}


\subsection{\texorpdfstring{Type 2: $w$ is not regular in $\mathbf{G}$}{Type 2}}\label{subsec:c2}
Let $\widetilde{\mathbf{G}}$ be our second type of triangular cactus graph $\mathbf{G}$ such that there exists at least one $3$-cycle in $\widetilde{\mathbf{G}}$ containing $w$, with the remaining vertices of this cycle having no other $3$-cycles attached to them.
Moreover, 
$\diam{\widetilde{\mathbf{G}}}=4$.

As earlier, we denote the vertices $x_i$ with no $3$-cycles $\{x_{i},y_{i,2t-1},y_{i,2t}\}$ attached to them as $\zeta_i$ and let there be $l$ such $\zeta_i$ in $\widetilde{\mathbf{G}}$.
Note that, according to the description of $\widetilde{\mathbf{G}}$, there exists at least a pair of vertices $(\zeta_{i},\zeta_{j})$ in $\widetilde{\mathbf{G}}$ such that $\{\zeta_{i},\zeta_{j}\}\in E(\widetilde{\mathbf{G}})$.  
Since $\diam{\widetilde{\mathbf{G}}}=4$, the number of vertices $\zeta_{i}\in V(\widetilde{\mathbf{G}})$ is at most $2(n-1)$. Therefore we have, $2\leq l\leq 2(n-1)$.
A generic diagram of $\widetilde{\mathbf{G}}$ is shown in Figure~\ref{fig:cactgen2}.

\begin{figure}[ht]
\centering
\begin{tikzpicture}
\draw[black, thin] (3,0) -- (5,2) -- (2,3)-- cycle;
\filldraw [black] (3,0) circle (0.8pt);
\filldraw [black] (2,3) circle (0.8pt);
\filldraw [black] (3.5,4.2) circle (0.8pt);
\filldraw [black] (6.5,4.2) circle (0.8pt);
\filldraw [black] (8,3) circle (0.8pt);
\filldraw [black] (7,0) circle (0.8pt);
\filldraw [black] (3,0) node[anchor=west] {$\zeta_i$};
\filldraw [black] (2,3) circle (0.8pt);
\filldraw [black] (2,3.2) node[anchor=west] {$\zeta_j$};
\draw[black, thin] (5,2) -- (3.5,4.2) -- (6.5,4.2)-- cycle;
\filldraw [black] (5,2) circle (1pt);
\filldraw [black] (5,2) node[anchor=north] {$w$};
\path (7,2.5) -- node[auto=false]{$\vdots$} (5,3.2);
\draw[black, thin] (5,2) -- (8,3) -- (7,0)-- cycle;
\draw[black, thin] (3.5,4.2) -- (2.6,4) -- (3,3.52)-- cycle;
\filldraw [black] (2.6,4) circle (0.5pt);
\filldraw [black] (3,3.52) circle (0.5pt);
\draw[black, thin] (3.5,4.2) -- (3.2,4.8) -- (3.8,5)-- cycle;
\filldraw [black] (3.2,4.8) circle (0.5pt);
\filldraw [black] (3.8,5) circle (0.5pt);
\path (3,4.5) -- node[auto=false]{$\vdots$} (3.5,4.5);
\draw[black, thin] (6.5,4.2) -- (6.6,3.4) -- (7,3.52)-- cycle;
\filldraw [black] (6.6,3.4) circle (0.5pt);
\filldraw [black] (7,3.52) circle (0.5pt);
\draw[black, thin] (6.5,4.2) -- (6.8,4.8) -- (7.5,4.8)-- cycle;
\filldraw [black] (6.8,4.8) circle (0.5pt);
\filldraw [black] (7.5,4.8) circle (0.5pt);
\path (6.7,4) -- node[auto=false]{$\vdots$} (7,4.5);
\draw[black, thin] (8,3) -- (7.8,3.4) -- (8.3,3.5)-- cycle;
\filldraw [black] (7.8,3.4) circle (0.5pt);
\filldraw [black] (8.3,3.5) circle (0.5pt);
\draw[black, thin] (8,3) -- (8.1,2.6) -- (8.5,2.4)-- cycle;
\filldraw [black] (8.1,2.6) circle (0.5pt);
\filldraw [black] (8.5,2.4) circle (0.5pt);
\path (8,2.7) -- node[auto=false]{$\vdots$} (8.5,3.7);
\draw[black, thin] (7,0) -- (6,-0.5) -- (6.5,-0.5)-- cycle;
\filldraw [black] (6,-0.5) circle (0.5pt);
\filldraw [black] (6.5,-0.5) circle (0.5pt);
\draw[black, thin] (7,0) -- (7.34,-0.5) -- (7.6,-0.3)-- cycle;
\filldraw [black] (7.34,-0.5) circle (0.5pt);
\filldraw [black] (7.6,-0.3) circle (0.5pt);
\path (6.36,-0.6) -- node[auto=false]{$\dots$} (7.6,0);
\end{tikzpicture}
\caption{An illustration of graph $\widetilde{\mathbf{G}}$}
\label{fig:cactgen2}
\end{figure}
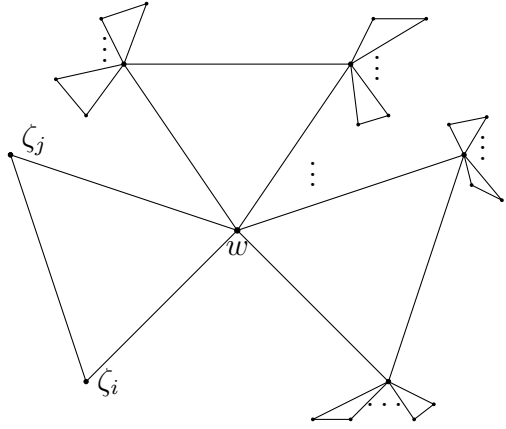

The set of cutpoints of $\widetilde{\mathbf{G}}$ is 
$\{w\}\cup \big(\{x_{i}\colon 1\leq i\leq 2n\}\backslash \bigcup_{i=1}^{l}\{\zeta_{i}\}\big) .$
Since there exists at least one $\{\zeta_{i},\zeta_{j}\}\in E(\widetilde{\mathbf{G}})$, we observe that $w$ is not regular in $\widetilde{\mathbf{G}}$. 
Note that, the only regular vertices in $\widetilde{\mathbf{G}}$ are the non-cutpoints of $\widetilde{\mathbf{G}}$. 

For $i\in [2n]$, let $s_i$ be the number of $3$-cycles, not containing the vertex $w$, attached to $x_{i}\in V(\widetilde{\mathbf{G}})$.
All the fundamental sets of $\widetilde{\mathbf{G}}$ containing the vertex $w$ have the form $\{w\}\cup\bigcup_{1\leq i \leq 2n}\{y_{i,k_{1}},y_{i,k_{2}},\dots, y_{i,k_{s_i}}\}$, where $k_{u}\in [2s_i]$ and none of the vertices $y_{i,k_{u}}$ are adjacent to each other.

Let us suppose that there are $m$ pairs of vertices $(\zeta_{p},\zeta_{p^{\prime}})$ in $\widetilde{\mathbf{G}}$ such that $\{\zeta_{p},\zeta_{p^{\prime}}\}\in E(\widetilde{\mathbf{G}})$.
For such a pair of vertices $(\zeta_{p},\zeta_{p^{\prime}})$, we define $\omega_{p}$ as $\omega_{p}\in \{\zeta_{p},\zeta_{p^{\prime}}\}$,  $1\leq p\leq m$.
In $\widetilde{\mathbf{G}}$, and for any $i,j\in [2n]$, the building blocks of any fundamental set not containing $w$, are as follows:
\begin{enumerate}[label=(\alph*)]
    \item \label{1a}The set $\bigcup_{p=1}^{m} \{\omega_{p}\}$ is fundamental in $\widetilde{\mathbf{G}}$.
    \item \label{2a} Suppose $s$ is the number of $3$-cycles that are attached to the vertex $x_{i}$, but do not contain the vertex $w$. Then the set $\{y_{i,k_{1}},y_{i,k_{2}},\dots, y_{i,k_{s}}\}$, where $k_{u}\in[2s]$ such that none of the vertices $y_{i,k_{u}}$ are adjacent to each other, forms a fundamental set in $\widetilde{\mathbf{G}}$.
    \item\label{3a} Let $\{x_{i},x_{j}\}\in E(\widetilde{\mathbf{G}})$ and $t$ be the number of $3$-cycles not containing $w$ that are attached to $x_{j}$. Then, $\bigcup_{p=1}^{m} \{\omega_{p}\} \cup \{x_{i}, y_{j,k_{1}},y_{j,k_{2}},\dots , y_{j,k_{t}}\}$ with none of the vertices $y_{j,k_{v}}$, $k_{v}\in[2t]$ adjacent to each other, is a fundamental set in $\widetilde{\mathbf{G}}$.
\end{enumerate}

Any fundamental set $T$ in $\widetilde{\mathbf{G}}$ with $w\notin T$ is either one of the listed building blocks or can be expressed as their unions.
For any fundamental set $T$ consisting of the building blocks in \ref{1a} or \ref{3a}, we observe that $w\in N_{\widetilde{\mathbf{G}}}(T)$.

Let us define some important notations similar to Section~\ref{subsec:c1}.
Suppose $p$ is a positive integer such that $1\leq p\leq \floor*{\frac{n}{2}}$, where $n$ is the total number of $3$-cycles with common vertex $w$.
We define
\begin{align*}
\mathcal{E}_{\widetilde{\mathbf{G}}}^{p}[(C, C^{\prime})]\\ 
& \hspace{-2cm}  \coloneqq 
\left\{(C_1, C_{1^\prime}),\dots, (C_p, C_{p^\prime})\colon 
\begin{aligned}
\phantom{a} & (C_i, C_{i^\prime}) \textrm{ is exceptional in } \widetilde{\mathbf{G}},\  \forall \ 1\leq i\leq p,\\
\phantom{b} & (C, C^{\prime}) =(C_{k}, C_{k^\prime}) \textrm{ for some } k\in[p], \textrm{ and }\\
\phantom{c} & (C_i, C_{j}) \textrm{ is also exceptional in } \widetilde{\mathbf{G}}  \textrm{ for any } i,j\in [p\cup p^{\prime}].
\end{aligned}
\right\}
\end{align*}
Further, for any set $\mathcal{E}_{\widetilde{\mathbf{G}}}^{p}[(C, C^{\prime})]$, we define 
$$\widetilde{q}_{\mathcal{E}_{p}}[(C, C^{\prime})]\coloneqq \sum\limits_{i=1}^{p}(\mathbb{E}_{C_{i}}+ \mathbb{E}_{C_{i^{\prime}}}),$$
where $(C_i, C_{i^\prime}) \in \mathcal{E}_{\widetilde{\mathbf{G}}}^{p}[(C, C^{\prime})]$, for all $1\leq i\leq p$.
Similar to that in Section~\ref{subsec:c1}, we know that $\widetilde{q}_{\mathcal{E}_{p}}[(C, C^{\prime})]\notin S_{\widetilde{\mathbf{G}}}$, for any $\mathcal{E}_{\widetilde{\mathbf{G}}}^{p}[(C, C^{\prime})]$.
In addition, we define $\mathcal{T}_{\mathcal{E}_{\widetilde{\mathbf{G}}}^{p}}[(C, C^{\prime})]$ as the set of all fundamental sets in $\widetilde{\mathbf{G}}$ such that we have $w\in N_{\widetilde{\mathbf{G}}}(T)$ and $\big(T\cup N_{\widetilde{\mathbf{G}}}(T)\big) \cap \big(V(C_i)\cup V(C_{i^{\prime}})\big) =\emptyset$ for any exceptional pair $(C_i, C_{i^{\prime}}) \in \mathcal{E}_{\widetilde{\mathbf{G}}}^{p}[(C, C^{\prime})]$. 
Note that, we may drop the notation $[(C, C^{\prime})]$ when the context is clear, referring simply to $\mathcal{T}_{\mathcal{E}_{\widetilde{\mathbf{G}}}^{p}}$.

\begin{prop}\label{prop:hole2}
Let $(C, C^{\prime})$ be any exceptional pair in $\widetilde{\mathbf{G}}$ and let $n$ be the total number of $3$-cycles in $\widetilde{\mathbf{G}}$ sharing the common vertex $w$.
For the triangular cactus graph $\widetilde{\mathbf{G}}$ (Figure~\ref{fig:cactgen2}), the set of holes $\overline{S_{\widetilde{\mathbf{G}}}}\backslash S_{\widetilde{\mathbf{G}}}$ is as follows:
\begin{equation}\label{eq:hole_cactus02}
\overline{S_{\widetilde{\mathbf{G}}}}\backslash S_{\widetilde{\mathbf{G}}} = \bigcup\limits_{\substack {\mathcal{E}_{\widetilde{\mathbf{G}}}^{p}[(C, C^{\prime})]; \\ \\ T\in\mathcal{T}_{\mathcal{E}_{\widetilde{\mathbf{G}}}^{p}}}} (\widetilde{q}_{\mathcal{E}_{p}}[(C, C^{\prime})]+\mathbb{Z}F_{T} ) \cap \mathcal{C}_{\widetilde{\mathbf{G}}}, 
\end{equation}
where the union is taken over all possible sets $\mathcal{E}_{\widetilde{\mathbf{G}}}^{p}[(C, C^{\prime})]$ for which $\widetilde{q}_{\mathcal{E}_{p}}[(C, C^{\prime})]$ is defined in $\widetilde{\mathbf{G}}$ and their corresponding fundamental sets $T\in\mathcal{T}_{\mathcal{E}_{\widetilde{\mathbf{G}}}^{p}}[(C, C^{\prime})]$, where $1\leq p\leq \floor*{\frac{n}{2}}$.
\end{prop}

\begin{proof}
For our graph $\widetilde{\mathbf{G}}$, we do not have any facet $F_w$ since the vertex $w$ is not regular. Therefore note that, for any $\beta\in S_{\widetilde{\mathbf{G}}}$, we have $\beta_{w}>0$.
Here, by observing the structure of the fundamental sets in $\widetilde{\mathbf{G}}$, the existence of edge $\{v, w\}$ in the formation of $\beta$ implies that $\beta$ belongs to some facets corresponding to the fundamental sets with building blocks of type \ref{1a} or \ref{3a}.
Further, we can proceed with a similar proof as that of Proposition~\ref{prop:hole1} and prove \eqref{eq:hole_cactus02}. 
\end{proof}

We combine all the above results and proceed to the proof of our main theorem.

\begin{proof}[\textbf{Proof of Theorem~\ref{thm:cactusS2}}]
Since $\diam{ \mathbf{G}} = 4$, there exists at least one pair of odd cycles of the form $(\{x_p, y_{p,2t-1},y_{p,2t}\},\ \{x_q, y_{q,2t^{\prime}-1},y_{q,2t^{\prime}}\})$ in $\mathbf{G}$ such that the distance between $x_{p}$ and $x_{q}$ is exactly $2$. 
All such pairs of odd cycles are exceptional in $\mathbf{G}$. Therefore, the graph $\mathbf{G}$ does not satisfy the odd cycle condition and by Theorem~\ref{thm:normal}, we conclude that the edge ring $\mathbb{K}[\mathbf{G}]$ is non-normal, irrespective of the two types of $\mathbf{G}$.

In Sections~\ref{subsec:c1} and \ref{subsec:c2}, we explored the two types of graph $\mathbf{G}$. 
Let us assume that both $\mathbf{G}^{\prime}$ and $\widetilde{\mathbf{G}}$ are triangular cactus graphs on $d$ vertices.
We denoted $F_T$ and $F_v$ as facets corresponding to the hyperplanes $\mathcal{H}_{T}$ and $\mathcal{H}_{v}$, defined by regular vertices $v$ and fundamental sets $T$, respectively.
Since the cone $\mathcal{C}_{\mathbf{G}}$ resides in a $d$-dimensional space, each facet $F_T$ and $F_v$ has dimension $d-1$.
This is because a supporting hyperplane removes one degree of freedom from the cone, effectively reducing its dimension by $1$.
Now, we compare \eqref{eq:hole_cactus01} and \eqref{eq:hole_cactus02} with the description of holes given in \eqref{eq:holes}.
We observe that, all the sets $q_{\mathcal{E}_{p}}[(C, C^{\prime})]+\mathbb{Z}F_{T}$ and $q_{\mathcal{E}_{p}}[(C, C^{\prime})]+\mathbb{Z}F_{w}$ in \eqref{eq:hole_cactus01}, as well as the sets $\widetilde{q}_{\mathcal{E}_{p}}[(C, C^{\prime})]+\mathbb{Z}F_{T}$ in \eqref{eq:hole_cactus02}, corresponds to a $(d-1)$-dimensional family of holes.
That is,  every family of holes of $S_{\mathbf{G}}$ is of dimension $d-1$.
Thus, by Theorem~\ref{thm:S2holes}, we say that the edge ring $\mathbb{K}[\mathbf{G}]$ always satisfies the $(S_2)$-condition.
\end{proof}

\section{Conclusions}\label{sec:conclude}

We have seen that the associated edge ring of a triangular cactus graph of diameter $4$ is always non-normal and satisfies the $(S_2)$-condition.
Moreover, our primary theorem, Theorem~\ref{thm:cactusS2}, provides supporting evidence to Conjecture~\ref{cnj:cactus}.

\begin{remark}\label{rmk:CMexpect2}
By further computations using {\tt Macaulay2}, we believe that every non-normal edge ring of triangular cactus graphs satisfying $(S_{2})$-condition is always Cohen--Macaulay. 
To the best of our knowledge, no triangular cactus graph has been identified whose edge ring does not satisfy the $(S_{2})$-condition.
Hence, we expect that the edge ring of every triangular cactus graph is Cohen--Macaulay. 
\end{remark}

\noindent
\textbf{Acknowledgments.} 
The first named author was supported by the Alexander von Humboldt Foundation and by a grant of the Ministry of Research, Innovation and Digitization, CNCS - UEFISCDI, project number PN-III-P1-1.1-TE-2021-1633, within PNCDI III.
The second named author expresses gratitude to Professor Akihiro Higashitani for his valuable comments, which greatly improved the manuscript. 
We want to extend our sincere gratitude to the anonymous reviewers for their invaluable feedback and insightful comments, which greatly enhanced the quality and clarity of this work.

\bibliography{CactusGraphs}

\end{document}